\documentclass[smallcondensed]{amsart}
\usepackage{amsmath,amssymb,amsthm,amsrefs}
\usepackage{xcolor}

\usepackage[utf8]{inputenc}

\newtheorem{teo}{Teorema}[section]
\newtheorem{prop}[teo]{Proposition}
\newtheorem{lema}[teo]{Lemma}
\newtheorem{cor}[teo]{Corollary}
\theoremstyle{definition}
\newtheorem{defn}[teo]{Definition}
\newtheorem{obs}[teo]{Remark}

\newtheorem{ex}[teo]{Example}

\newcommand{\nd}{\noindent}

\begin{document}
	\allowdisplaybreaks

\title{Partial Coactions of Weak Hopf Algebras on Coalgebras}

\author[Fonseca]{Graziela Fonseca}
\address[Fonseca]{Instituto Federal Sul-Riograndense, Brazil} 
\email{grazielafonseca@charqueadas.ifsul.ebu.br}

\author[Fontes]{Eneilson Fontes}
\address[Fontes]{Universidade Federal do Rio Grande, Brazil}
\email{eneilsonfontes@furg.br}

\author[Martini]{Grasiela Martini}
\address[Martini]{Universidade Federal do Rio Grande, Brazil}
\email{grasiela.martini@furg.br}

\begin{abstract}
It will be seen that if $H$ is a weak Hopf algebra in the definition of coaction of weak bialgebras on coalgebras \cite{Wang}, then a definition property is suppressed giving rise to the (global) coactions of weak Hopf algebras on coalgebras. The next step will be introduce the more general notion of partial coactions of weak Hopf algebras on coalgebras as well as a family of examples via a fixed element on the weak Hopf algebra, illustrating both definitions: global and partial. Moreover, it will also be presented how to obtain a partial comodule coalgebra from a global one via projections, giving another way to find examples of partial coactions of weak Hopf algebras on coalgebras. In addition, the weak smash coproduct \cite{Wang} will be studied and it will be seen under what conditions it is possible to generate a weak Hopf algebra structure from the coproduct and the counit defined on it. Finally, a dual relationship between the structures of partial action and partial coaction of a weak Hopf algebra on a coalgebra will be established.
\end{abstract}

\thanks{The first author was partially supported by CNPq, Brazil}

\maketitle

{\nd\scriptsize{\bf Key words:} Weak Hopf algebra, globalization, dualization, partial comodule coalgebra, weak smash coproduct.}	\\
{\nd\scriptsize{\bf Mathematics Subject Classification:} primary 16T99; secondary 20L05}

\section{Introduction}

Partial action theory appeared firstly in \cite{Exelp} in the context of operator algebra. Later, in \cite{Dokuchaev}, M. Dokuchaev and R. Exel brought partial actions to a purely algebraic context contributing to the development of classical results, such as Galois theory, in the case of partial actions of groups on rings.

Following this line of research, S. Caenepeel and K. Janssen introduced the notions of partial actions and coactions of Hopf algebras on algebras in \cite{CaenJanssen}. The main idea of studying partial actions for the context of Hopf algebras is to generalize the results obtained for partial group actions to this broader context. The notions of partial actions and coactions of Hopf algebras on coalgebras appeared for the first time in \cite{Glauberglobalizations}, dualizing the structures introduced in \cite{CaenJanssen}.

As a natural task, in \cite{Felipeweak}, was introduced the notion of partial actions of weak Hopf algebras on algebras. In this work, the authors extended many results of the classic theory for this setting.

We introduced in \cite{EGG} the theory of partial actions of weak Hopf algebras on coalgebras, inspired by the notion of partial action of a Hopf algebra on a coalgebra, presented in \cite{Glauberglobalizations}. Basically, it was constructed in \cite{EGG} a correspondence between a partial action of a groupoid $\mathcal{G}$ on a coalgebra $C$ and a partial action of the groupoid algebra $\Bbbk \mathcal{G}$ on the coalgebra $C$. 

In the present work, we give successions to the theory of partial actions. The notion of partial and global coactions of weak Hopf algebras on coalgebras is introduced as well as some important properties and examples. In the sequel, we will study the weak smash coproduct presented in \cite{Wang} in order to see under what conditions this structure is a weak Hopf algebra. We divide this paper as follows:

The second section is devoted to the study of weak Hopf algebras, their properties and some examples that will be commonly used throughout the text. A weak bialgebra is a vector space that has a structure of algebra and coalgebra simultaneously, with a compatibility property between these structures. The axioms of weak bialgebra appear for the first time in \cite{Bohm}. If a weak bialgebra is provided with an anti-homomorphism of algebras and coalgebras, them we say that is a weak Hopf algebra. The main difference between a weak Hopf algebra and a Hopf algebra is that in the case of a Hopf algebra the counit is an algebra homomorphism.

The concept of coaction of a weak bialgebra on a coalgebra was introduced in \cite{Wang}. In section 3, the coaction of a weak Hopf algebra on a coalgebra is presented. Generalizing this concept, the definition of partial coaction of a weak Hopf algebra on a coalgebra is exhibited with its properties and a family of examples. It is also ascertained what conditions are necessary and sufficient for a partial comodule coalgebra to be generated from a global comodule coalgebra via a projection. 

Section 4 is intended to investigate the weak smash coproduct presented in \cite{Wang}. The idea is to construct a weak Hopf algebra from the existing coalgebra structure in the weak smash coproduct. Historically, the construction of Hopf algebras and weak Hopf algebras from global and partial (co)actions has been studied by several authors. This can be seen in texts such as \cite{NASN}, \cite{Majid} and \cite{Takeuchi}. This shows a great concern in presenting new examples of such structures. Our contribution is to make the weak smash coproduct into a weak Hopf algebra under certain conditions.


From now, some notations will be fixed. It will be denoted by $\Bbbk$ a generic field, unless some additional specification is made about such structure. Moreover, every tensorial product will be considered over the field $\Bbbk$, then, it will be used the notation $\otimes$ instead of $\otimes_{\Bbbk}$. $A$ will always denote an algebra, $C$ a coalgebra and $H$ a weak Hopf algebra. Throughout the text other properties may be required over the structures $A$, $C$ e $H$, but they will be duly mentioned. Besides that, every map will be considered $\Bbbk$-linear and the vector spaces will be considered over the field $\Bbbk$. Finally, the isomorphism $V \otimes \Bbbk \simeq V \simeq \Bbbk \otimes V$ will be used automatically for every vector space $V$.

\section{Preliminaries}
In this section, we present few results of weak Hopf algebras. For more details we refer \cite{Bohminicio}, \cite{Bohmexemplo} and \cite{Bohm}.

A \textit{weak bialgebra} $(H, m, u, \Delta, \varepsilon)$ (or simply $H$) is a vector space such that $ (H, m, u) $ is an algebra, $(H,\Delta,\varepsilon)$ is a coalgebra, and, in addition, the following conditions are satisfied for all $h,k \in H$:
\begin{enumerate}
	\item[(i)] $\Delta(hk) = \Delta(h)\Delta(k)$;
	\item[(ii)] $\varepsilon(hk\ell) = \varepsilon(hk_1)\varepsilon(k_2\ell) = \varepsilon(hk_2)\varepsilon(k_1\ell)$;
	\item[(iii)] $(1_H \otimes \Delta(1_H))(\Delta(1_H) \otimes 1_H) = (\Delta(1_H) \otimes 1_H)(1_H \otimes \Delta(1_H)) = \Delta^2(1_H)$.
\end{enumerate}

Since $\Delta$ is multiplicative, we conclude that $\Delta(h) = \Delta (h1_H) = \Delta(1_Hh)$, then
\begin{eqnarray}
h_1 \otimes h_2 = h_11_1 \otimes h_21_2 
= 1_1h_1 \otimes 1_2h_2. \label{4.1}
\end{eqnarray}
It is possible to use $ \varepsilon $ to define the following linear maps
\begin{eqnarray*}
	\varepsilon_t: H &\rightarrow& H \ \ \ \ \ \ \  \ \ \ \ \mbox{ and} \ \  \ \ 	\varepsilon_s: H \rightarrow H \\
	h &\mapsto& \varepsilon(1_1h)1_2 \ \ \ \ \ \ \ \ \ \ \ \  \ \ \ \  \ \ \	h \mapsto 1_1\varepsilon(h1_2) .
\end{eqnarray*}
Then, it can be defined the vector spaces $H_t= \varepsilon_t(H)$ and $H_s= \varepsilon_s(H)$. Thus, for any weak bialgebra $H$, every element $ h \in H $ can be written as
\begin{eqnarray}
h = (\varepsilon \otimes I) \Delta(h) 
= (\varepsilon \otimes I) \Delta(1_H h) 
= \varepsilon_t(h_1)h_2,\label{propt}\\
h = (I \otimes \varepsilon) \Delta(h) 
= (I \otimes \varepsilon) \Delta(h1_H) 
= h_1 \varepsilon_s(h_2). \label{props}
\end{eqnarray}

\begin{prop}
Let $H$ be a weak bialgebra. Then, the following properties hold for all $h, k \in H$
\begin{eqnarray}
\varepsilon_t(\varepsilon_t(h)) &=& \varepsilon_t(h)  \label{4.3}\\
\varepsilon_s(\varepsilon_s(h)) &=& \varepsilon_s(h)  \label{4.4} \\
\varepsilon(h\varepsilon_t(k)) &=& \varepsilon(hk)  \label{4.5} \\
\varepsilon(\varepsilon_s(h)k) &=& \varepsilon(hk)  \label{4.6} \\
\Delta(1_H) &\in& H_s \otimes H_t \label{4.7} \\
\varepsilon_t(h\varepsilon_t(k)) &=& \varepsilon_t(hk)  \label{4.8}\\
\varepsilon_s(\varepsilon_s(h)k) &=& \varepsilon_s(hk)  \label{4.9} \\
\Delta(h) &=& 1_1h \otimes 1_2 \mbox{ for all $h \in H_t$}  \label{4.10}\\
\Delta(h) &=& 1_1 \otimes h1_2  \mbox{ for all $h \in H_s$}\label{4.11}\\
h_1 \otimes \varepsilon_t(h_2) &=& 1_1h \otimes 1_2  \label{4.12}\\
\varepsilon_s(h_1) \otimes h_2 &=& 1_1 \otimes h1_2 \label{4.13} \\
h\varepsilon_t(k) &=& \varepsilon(h_1k)h_2  \label{4.14} \\
\varepsilon_s(h)k &=& k_1\varepsilon(hk_2).  \label{4.15}
\end{eqnarray}
\end{prop}

Therefore, $ H_t $ and $ H_s $ are subalgebras of $ H $ such that contain $ 1_H $ and 
\begin{eqnarray}
hk &=& kh  \mbox{ for all $h \in H_t \ \mbox{and} \ k \in H_s$}. \label{4.16}
\end{eqnarray}
Finally, it is still possible to show that
\begin{eqnarray}
\varepsilon_t(\varepsilon_t(h)k) &=& \varepsilon_t(h)\varepsilon_t(k)  \label{4.19} \\
\varepsilon_s(h\varepsilon_s(k)) &=& \varepsilon_s(h)\varepsilon_s(k),  \label{4.20}
\end{eqnarray}
for all $h, k \in H$.

Let $ H $ be a weak bialgebra. We say that $ H $ is a \textit{weak Hopf algebra} if there is a linear map $ S: H\longrightarrow H$, called \textit{antipode}, which satisfies:	\begin{enumerate}
		\item [(i)] $h_1S(h_2)=\varepsilon_t(h)$;
		\item [(ii)] $S(h_1)h_2=\varepsilon_s(h)$;
		\item [(iii)] $S(h_1)h_2S(h_3)=S(h),$
	\end{enumerate}
	for all $h \in H$. The antipode of a weak Hopf algebra is  \textit{anti-multiplicative}, that is, $ S (hk) = S (k) S (h) $, and \textit{anti-comultiplicative}, which means $ S (h) _1 \otimes S (h) _2 = S (h_2) \otimes S (h_1) $. 
	
\begin{prop}
Let $H$ be a weak Hopf algebra. Then, the following identities hold for all $h \in H$
\begin{eqnarray}
\varepsilon_t(h) &=& \varepsilon(S(h)1_1)1_2 \label{4.30}\\
\varepsilon_s(h) &=& 1_1\varepsilon(1_2S(h)) \label{4.31}\\
\varepsilon_t \circ S &=& \varepsilon_t \circ \varepsilon_s = S \circ \varepsilon_s \label{4.34}\\
\varepsilon_s \circ S &=& \varepsilon_s \circ \varepsilon_t = S \circ \varepsilon_t \label{4.35}\\
h_1 \otimes S(h_2)h_3 &=& h1_1 \otimes S(1_2) \  \label{4.38}\\
h_1S(h_2) \otimes h_3 &=& S(1_1) \otimes 1_2h. \  \label{4.39}
\end{eqnarray}
\end{prop}

Hence, if $ H $ is a weak Hopf algebra,  $S(1_H) = 1_H$, $\varepsilon \circ S = \varepsilon$, $S(H_t)=H_s$, $S(H_s)=H_t$ and $ S (H) $ is also a weak Hopf algebra, with the same counit and antipode. It is  easy to see that every Hopf algebra is a weak Hopf algebra. Conversely we have the following result.
\begin{prop}
	A weak Hopf algebra is a Hopf algebra if one of the following equivalent conditions is satisfied:
	\begin{itemize}
		\item [(i)] $\Delta(1_H)=1_H \otimes 1_H$;
		\item[(ii)] $\varepsilon(hk) = \varepsilon(h)\varepsilon(k);$
		\item [(iii)]$h_1S(h_2)=\varepsilon(h)1_H; $
		\item [(iv)]$S(h_1)h_2 = \varepsilon(h)1_H;$
		\item [(v)]$H_t=H_s=\Bbbk 1_H;$
	\end{itemize}
	for all $h, k \in H$.	
\end{prop}

In order to construct an example of weak Hopf algebra, we present the following definition.

\begin{defn}[Groupoid] \label{grupoide}
	Consider $ \mathcal{G} $ a non-empty set with a binary operation partially defined which is denoted by concatenation. This operation is called product. Given $ g, h \in \mathcal{G} $, we write $ \exists gh $ whenever the product $ gh $ is set (similarly we use $ \nexists gh $ whenever the product is not defined). Thus, $ \mathcal{G} $ is called \textit{groupoid} if:
	
	\begin{itemize}
		\item [(i)] For all $g, h, l \in \mathcal{G}$, $\exists(gh)l$ if and only if $\exists g(hl)$, and, in this case, $(gh)l = g(hl)$;
		\item [(ii)] For all $g, h, l \in \mathcal{G}$, $\exists(gh)l$ if and only if $\exists gh$ and $\exists hl$;
		\item[(iii)] For each $g \in \mathcal{G}$ there are unique elements $d(g), r(g) \in \mathcal{G}$ such that $\exists gd(g)$, $\exists r(g)g$ and $gd(g)=g=r(g)g$;
		\item[(iv)] For each $g \in \mathcal{G}$ there exists an element such that $d(g)= g^{-1}g$ and $r(g)=gg^{-1}$.
	\end{itemize}
\end{defn}

Moreover, the element $ g^{-1}$ is the only one that satisfies such property and, in addition, $({g^{-1}})^{-1} = g$, for all $g \in \mathcal{G} $. An element $ e $ is said identity in $ \mathcal{G} $ if for some $ g \in \mathcal{G} $, $ e = d (g) = r (g^{-1})$. Therefore, $ e^2 = e $, which implies that $ d (e) = e = r (e) $ and $ e = e^{-1}$. We denote $ \mathcal{G}_0 $ the set of all identities elements of $ \mathcal{G} $. Besides that, one can define the set
$\mathcal{G}^{2} = \{(g,h) \in \mathcal{G} \times \mathcal{G} \ | \ \exists gh\}$ of all pairs of elements composable in $\mathcal{G}$.

\begin{prop}
	Let $\mathcal{G}$ be a groupoid. Then, for all $g, h \in \mathcal{G}$:
	\begin{itemize}
		\item [(i)] $\exists gh$ if and only if $d(g)=r(h)$ and, in this case, $d(gh) = d(h)$ and $r(gh)=r(g)$;
		\item [(ii)]  $\exists gh$ if and only if $\exists h^{-1} g^{-1}$ and, in this case, ${(gh)}^{-1} = h^{-1} g^{-1}$.
	\end{itemize}
\end{prop}
\begin{ex}[Groupoid Algebra] \label{algebradegrupoide} Let $\mathcal{G}$ be a groupoid such that the cardinality of $\mathcal{G}_{0}$ is finite and $ \Bbbk \mathcal{G}$ the vector space with basis indexed by the elements of $\mathcal{G}$ given by $\lbrace \delta_g \rbrace_{g\in \mathcal{G}}$. Then, $\Bbbk \mathcal{G}$ is a weak Hopf algebra with the following structures
	\begin{eqnarray*}
		m(\delta_g\otimes \delta_h)=\left\{
		\begin{array}{rl}
			\delta_{gh}, & \text{if $\exists gh$ },\\
			0, & \text{ otherwise }
		\end{array} \right. \ \  \ \ \ \ u(1_{\Bbbk})=1_\mathcal{G} = \sum_{e \in \mathcal{G}_{0}} \delta_e
	\end{eqnarray*}	
	$$ \Delta(\delta_g)=\delta_g\otimes \delta_g \ \ \ \ \ \ \ \ \ \ \ \ \ \ \ \  \ 
	\varepsilon(\delta_g)=1_{\Bbbk} \ \ \ \ \ \ \ \ \ \ \ \ \ \ \ \  \  S(\delta_g)=\delta_{g^{-1}}.$$
	
	\label{ex_grupoide}
\end{ex}

Remark that when it is assumed that the dimension of a weak Hopf algebra $ H $ is finite, it is obtained that the dual structure $H^{*}=Hom({H, \Bbbk})$ is a weak Hopf algebra with the convolution product $(f*g)(h)=m(f\otimes g)(h)=f(h_1)g(h_2)$ for all $h \in H,$ the unit $u_{H^*}(1_{\Bbbk})=1_{H^*}= \varepsilon_H$, the coprodut defined by the relation $\Delta_{H^*}(f) = f_1 \otimes f_2 \Leftrightarrow f(hk)=f_1(h)f_2(k)$ for all $h,k \in H$, and the counit $\varepsilon_{H^*}(f)=f(1_H).$
Besides that, we have
$(\varepsilon_t)_{_{H^*}}(f)= f \circ \varepsilon_t$
and
$(\varepsilon_s)_{_{H^*}}(f)= f \circ \varepsilon_s.$

\begin{ex} [Dual Groupoid Algebra] Let $\mathcal{G}$ be a finite groupoid and $ (\Bbbk \mathcal{G})^{*}$  the  vector space with basis indexed by the elements of $\mathcal{G}$ given by $\lbrace p_g \rbrace_{g\in \mathcal{G}}$, where
	\begin{eqnarray*}
		p_g(\delta_h)=\left\{
		\begin{array}{rl}
			1_{\Bbbk}, & \text{if $g=h$ },\\
			0, & \text{ otherwise. }
		\end{array} \right.
	\end{eqnarray*}	
	
	Then, $(\Bbbk \mathcal{G})^{*}$ is a weak Hopf algebra with the following structures
	\begin{eqnarray*}
		p_g * p_h=\left\{
		\begin{array}{rl}
			p_g, & \text{if $g=h$ },\\
			0, & \text{ otherwise }
		\end{array} \right.  \ \ \ \ \ \ \  \ \ 	1_{(\Bbbk \mathcal{G})^{*}} = \sum_{g \in \mathcal{G}} p_g 
	\end{eqnarray*}	
	$$\Delta_{_{(\Bbbk \mathcal{G})^{*}}}(p_g)=\sum_{h \in \mathcal{G}, \exists h^{-1}g} p_h \otimes p_{h^{-1}g} \ \ \ \ \ \ \ \ \varepsilon_{_{(\Bbbk \mathcal{G})^{*}}}(p_g)=p_g(1_{\Bbbk \mathcal{G}}) \ \ \ \ \ \ \ \ S_{_{(\Bbbk \mathcal{G})^{*}}}(p_g)=p_{g} \circ S.$$
	
	\label{ex_grupoidedual}
\end{ex}

The following example of weak Hopf algebra was presented by G. Böhm and J. Gómes-Torrecillas in \cite{Bohmexemplo}.

\begin{ex}\label{exemplodogrupo}
	Consider $G$ a finite abelian group with cardinality $|G|$, where $|G|$  is not a multiple of the characteristic of $\Bbbk$. If we consider $\Bbbk G$ the algebra with basis indexed by the elements of $G$ and with coalgebra structure given by
	
	$$\Delta(g)=\frac{1}{|G|}\sum_{h \in {G}} gh  \otimes h^{-1} \ \ \ \ \varepsilon(g)=\left\{
	\begin{array}{rl}
	|G|, & \text{if $g=1_G$ },\\
	0, & \text{ otherwise. }
	\end{array} \right.
	$$
	Then, $\Bbbk G$ is a weak Hopf algebra with antipode defined by $S(g)=g$. Besides that, $\varepsilon_s(g)=\varepsilon_t(g)=g$, for all $g \in G$, what implies that $H_s=H_t=\Bbbk G$.
	
\end{ex}

\section{Comodule Coalgebra} 
Consider $H$ a weak bialgebra. In \cite{Wang}, Yu. Wang and L. Zhang defined $C$ a (left) \textit{$H$-comodule coalgebra}  when  there exits a linear map 
\begin{eqnarray*}
	\rho: C & \longrightarrow & H\otimes C\\
	c & \longmapsto & c^{-1} \otimes c^0
\end{eqnarray*}
such that for all $c\in C$
\begin{enumerate}
	\item [(CC1)] $(\varepsilon_H \otimes I_C)\rho(c)=c$
	\item [(CC2)] $(I_H \otimes \Delta_C)\rho(c) = (m_H \otimes I_C \otimes I_C)(I_H \otimes \tau_{C , H} \otimes I_C)(\rho \otimes \rho)\Delta_C(c)$
	\item [(CC3)] $(I_H \otimes \rho)\rho(c) = (\Delta_H \otimes I_C)\rho(c)$
	\item[(CC4)] $(I_H \otimes \varepsilon_C) \rho(c)= (\varepsilon_t \otimes \varepsilon_C) \rho(c).$
\end{enumerate}

In this case, it is said that $H$ coacts on the coalgebra $C$. 

\begin{prop}\label{caracrho}
Let $H$ be a weak Hopf algebra. If there exists a linear map
	\begin{eqnarray*}
		\rho: C & \longrightarrow & H\otimes C\\
		c & \longmapsto & c^{-1} \otimes c^0
	\end{eqnarray*}
that satisfies (CC1)-(CC3), then the condition (CC4) is satisfied.
\end{prop}

\begin{proof}
Suppose that there is a linear map $\rho$ that  satisfies the conditions (CC1)-(CC3), then for every $c \in C$
	\begin{eqnarray*}
		\varepsilon_t(c^{-1})\varepsilon_C(c^0) &=& {c^{-1}}_1S_H({c^{-1}}_2)\varepsilon_C(c^0) \\
		&\stackrel{(CC3)}{=}& {c^{-1}}S_H({c^{0-1}})\varepsilon_C(c^{00})\\
		&\stackrel{(CC2)}{=}& {c_1}^{-1}{{c_2}^{-1}}S_H({{c_2}^{0-1}}) \varepsilon_C({c_2}^{00})\varepsilon_C({c_1}^{0}) \\
		&\stackrel{(CC3)}{=}&{c_1}^{-1}{{c_2}^{-1}}_1S_H({{c_2}^{-1}}_2) \varepsilon_C({c_1}^{0}) \varepsilon_C({c_2}^0) \\
		&\stackrel{(\ref{4.14})}{=}&{{c_1}^{-1}}_2 \varepsilon_H({{c_1}^{-1}}_1{c_2}^{-1}) \varepsilon_C({c_1}^0) \varepsilon_C({c_2}^0) \\
		&\stackrel{(CC3)}{=}&{c_1}^{0-1} \varepsilon_C({c_1}^{00}) \varepsilon_C({c_2}^0) \varepsilon_H({c_1}^{-1}{c_2}^{-1})\\
		&\stackrel{(CC2)}{=}&{{c^0}_1}^{-1} \varepsilon_C({{c^0}_1}^0) \varepsilon_C({c^0}_2) \varepsilon_H(c^{-1})\\
		&\stackrel{(CC1)}{=}& {c_1}^{-1} \varepsilon_C({c_1}^0)\varepsilon_C(c_2)\\
			&=&(I_H \otimes \varepsilon_C)\rho({c}).
	\end{eqnarray*}
\end{proof}

\begin{ex} \cite{Wang}
	Consider $H$ a weak Hopf algebra finite dimensional. Then, $H$ is a $H^*$-comodule coalgebra defined by
	\begin{eqnarray*}
		\rho: H & \longrightarrow & H^* \otimes H\\
		h_i & \longmapsto & {h_i}^* \otimes h_i
	\end{eqnarray*}
	where $\{h_i\}_{i=1}^{n}$ is a basis for $H$ and $\{{h_i}^*\}_{i=1}^{n}$ is the dual basis for $H^*$.
\end{ex}

\subsection{Partial Comodule Coalgebra}
\quad
In this section, the main purpose is to introduce the concept of a partial coaction of a weak Hopf algebra $H$ on a coalgebra. It is also introduced some examples that support the theory exposed here and some properties.

\begin{defn}
We say that $C$ is a (left) partial $H$-comodule coalgebra (or that $H$ coacts partially on $C$) if there exists a linear map
	$$\begin{array}{rl}
	\overline{\rho}: C & \longrightarrow  H\otimes C\\
	c & \longmapsto c^{\overline{-1}} \otimes c^{\overline{0}}
	\end{array}$$
such that for all $c \in C$
	
	\begin{enumerate}
		\item [(\label{CCP1}CCP1)] $(\varepsilon_H \otimes I_C)\overline{\rho}(c)=c$
		\item [(\label{CCP2}CCP2)] $(I_H \otimes \Delta_C)\overline{\rho}(c) = (m_H \otimes I_C \otimes I_C)(I_H \otimes \tau_{C , H} \otimes I_C)(\overline{\rho} \otimes \overline{\rho})\Delta_C(c)$
		\item [(\label{CCP3}CCP3)] $(I_H \otimes \overline{\rho}) \overline{\rho}(c) = (m_H \otimes I_H \otimes I_C)[(I_H \otimes \varepsilon_C)(\overline{\rho}(c_1)) \otimes (\Delta_H \otimes I_C)(\overline{\rho}(c_2))]$.
	\end{enumerate}
\end{defn}

Moreover, $C$ is said a (left) \textit{symmetric partial $H$-comodule coalgebra} if, in addition, satisfies
$$(I_H \otimes \overline{\rho}) \overline{\rho}(c) = (m_H \otimes I_H \otimes I_C)(I_H \otimes \tau_{H \otimes C, H})[(\Delta_H \otimes I_C)(\overline{\rho}(c_1)) \otimes (I_H \otimes \varepsilon_C)(\overline{\rho}(c_2))].$$

 \begin{obs}
 	Every $H$-comodule coalgebra is a partial $H$-comodule coalgebra. Indeed for every $c \in C$
 \begin{eqnarray*}
	(I_H \otimes {\rho}) {\rho}(c)&=&c^{{-1}} \otimes c^{{0-1}} \otimes c^{00}\\
	&\stackrel{(CC3)}{=}&{c^{-1}}_1 \otimes {c^{-1}}_2 \otimes c^0\\
	&\stackrel{(CC2)}{=}&{{c_1}^{-1}}_1 {{c_2}^{-1}}_1\otimes {{c_1}^{-1}}_2 {{c_2}^{-1}}_2 \otimes {c_2}^0 \varepsilon_C({c_1}^0)\\
	&\stackrel{\ref{caracrho}}{=}&{\varepsilon_t({c_1}^{-1})}_1 {{c_2}^{-1}}_1\otimes {\varepsilon_t({c_1}^{-1})}_2 {{c_2}^{-1}}_2 \otimes {c_2}^0 \varepsilon_C({c_1}^0)\\
	&\stackrel{(\ref{4.10})}{=}&{1_H}_1\varepsilon_t({c_1}^{-1}) {{c_2}^{-1}}_1\otimes {1_H}_2 {{c_2}^{-1}}_2 \otimes {c_2}^0 \varepsilon_C({c_1}^0)\\
	&\stackrel{(\ref{4.16})}{=}&\varepsilon_t({c_1}^{-1}) {1_H}_1{{c_2}^{-1}}_1\otimes {1_H}_2 {{c_2}^{-1}}_2 \otimes {c_2}^0 \varepsilon_C({c_1}^0)\\
	&\stackrel{(\ref{4.1})}{=}&{c_1}^{-1} {{c_2}^{-1}}_1\otimes {{c_2}^{-1}}_2 \otimes {c_2}^0 \varepsilon_C({c_1}^0)\\
	&=& (m_H \otimes I_H \otimes I_C)[(I_H \otimes \varepsilon_C)({\rho}(c_1)) \otimes (\Delta_H \otimes I_C)({\rho}(c_2))].
\end{eqnarray*}
\end{obs}

\begin{prop}\label{caraccomod}
Let $C$ be a partial $H$-comodule coalgebra. Then, $C$ is a $H$-comodule coalgebra if and only if $c^{\overline{-1}}\varepsilon_C(c^{\overline{0}}) = \varepsilon_t(c^{\overline{-1}})\varepsilon_C(c^{\overline{0}})$ for all $c\in C$.
\end{prop}

\begin{proof}
	Suppose that $C$ is a partial $H$-comodule coalgebra that satisfies
	$c^{\overline{-1}}\varepsilon_C(c^{\overline{0}}) = \varepsilon_t(c^{\overline{-1}})\varepsilon_C(c^{\overline{0}})$, then it is enough to show that
	$(I_H \otimes \overline{\rho})\overline{\rho}(c) = (\Delta_H \otimes I_C)\overline{\rho}(c)$:
	\begin{eqnarray*}
		(I_H \otimes \overline{\rho})\overline{\rho}(c) 		&\stackrel{(\ref{4.1})}{=}&{c_1}^{\overline{-1}}\varepsilon_C({c_1}^{\overline{0}}){1_H}_1{{c_2}^{\overline{-1}}}_1\otimes {1_H}_2{{c_2}^{\overline{-1}}}_2 \otimes {c_2}^{\overline{0}}\\
		&\stackrel{(\ref{4.16})}{=}&{1_H}_1\varepsilon_t(c^{\overline{-1}})\varepsilon_C({c_1}^{\overline{0}}){{c_2}^{\overline{-1}}}_1\otimes {1_H}_2{{c_2}^{\overline{-1}}}_2 \otimes {c_2}^{\overline{0}}\\
		&\stackrel{(\ref{4.10})}{=}&{\varepsilon_t(c^{\overline{-1}})}_1\varepsilon_C({c_1}^{\overline{0}}){{c_2}^{\overline{-1}}}_1\otimes {\varepsilon_t(c^{\overline{-1}})}_2{{c_2}^{\overline{-1}}}_2 \otimes {c_2}^{\overline{0}}\\
		&\stackrel{(CCP2)}{=}&{c^{\overline{-1}}}_1\varepsilon_C({c^{\overline{0}}}_1)\otimes {c^{\overline{-1}}}_2 \otimes {c^{\overline{0}}}_2\\
		&=&(\Delta_H \otimes I_C)\overline{\rho}(c).
	\end{eqnarray*}
\end{proof}
\begin{ex}
Consider $\Bbbk \mathcal{G}$ a groupoid algebra, where $\mathcal{G}$ is generated by the disjoint union of the finite groups $G_1$ and $ G_2$. Therefore, the group algebra $\Bbbk G_1$ is a partial $(\Bbbk \mathcal{G})^*$-comodule coalgebra via
	\begin{eqnarray*}
		\rho:\Bbbk G_1 &\rightarrow& (\Bbbk \mathcal{G})^{*} \otimes \Bbbk G_1\\
		h &\mapsto& p_{e_1} \otimes h,
	\end{eqnarray*}	
where $e_1$ is the identity  element of $G_1$.
\end{ex}

\subsection{Coactions via $\rho_h$}
\quad 
In this section it is explored a specific family of examples of partial comodule coalgebra. We say  that $C$ is a (left) \textit{$H$-comodule coalgebra via $\rho_h$} if, for some $h \in H$ fixed, the linear map
\begin{eqnarray*}
	\rho_h: C & \longrightarrow & H\otimes C\\
	c & \longmapsto & h \otimes c
\end{eqnarray*}
defines a structure of comodule coalgebra on $C$. Note that since $H$ is a weak Hopf algebra, the above application does not always defines a structure of $H$-comodule coalgebra on $C$. To see this, it is enough to observe that $\rho_{1_H}(c) = 1_H \otimes c$ turns out $C$ on a comodule coalgebra if and only if $H$ is a Hopf algebra. The following result has the intention to characterize the properties that an element $h \in H$ must to satisfy in order that $\rho_h$ be a coaction of $H$ on a coalgebra $C$.

\begin{prop}\label{rho_h}
We say that	$C$ is a (left) $H$-comodule coalgebra via $\rho_h$ if and only if
	\begin{enumerate}
		\item [(i)] $\varepsilon_H(h)=1_{\Bbbk}$
		\item[(ii)] $h^2=h$
		\item [(iii)] $\Delta_H(h)=h \otimes h.$
	\end{enumerate}
\end{prop}

\begin{proof}
The proof follows immediately from the definition comodule coalgebra via $\rho_h$.
\end{proof}

Besides that, if $h \in H$ satisfies the properties of Proposition \ref{rho_h}, then $h = \varepsilon_t(h)$.

\begin{ex}
	Consider $\Bbbk \mathcal{G}$ the groupoid algebra generated by a groupoid $\mathcal{G}$. Thus, fixing an element $e \in \mathcal{G}_0$, 
$\rho_{\delta_e}$ ensure a structure of $\Bbbk \mathcal{G}$-comodule coalgebra on any coalgebra $C$ by Proposition \ref{rho_h}.
\end{ex}

Moreover, it is easy to see that $\rho_{\delta_g}$ gives a structure of $H$-comodule coalgebra on a coalgebra $C$ if and only if $g \in \mathcal{G}_0$.

\begin{ex}
	Consider $\Bbbk {G}$ the weak Hopf algebra given in Example \ref{exemplodogrupo}, $h$ an element in $G$ and $C$ a coalgebra. $C$ is a $\Bbbk {G}$-comodule coalgebra via $\rho_h$ if and only if $G=\{1_{G}\}$.
\end{ex}

Thinking on the partial case, we say that $C$ is a (left) \textit {partial $ H$-comodule coalgebra via $\overline {\rho_h}$}, for some fixed $h \in H $, if the linear map
\begin{eqnarray*}
	\overline{\rho_h}: C & \longrightarrow & H\otimes C\\
	c & \longmapsto & h \otimes c
\end{eqnarray*}
determines a structure of partial $H$-comodule coalgebra on $C$.

\begin{prop}\label{rho_hparcial}
	$C$ is a (left) partial $H$-comodule coalgebra via $\overline{\rho_h}$ if and only if
	\begin{itemize}
		\item [(i)] $\varepsilon_H(h)=1_{\Bbbk}$
		\item [(ii)] $(h \otimes 1_H)\Delta_H(h)=h \otimes h.$
	\end{itemize}
\end{prop}

Observe that if $h \in H$ satisfies the properties (i) and (ii), then $h^2=h$.

\begin{proof}
The proof follows immediately from the definition of partial $ H$-comodule coalgebra via $\overline {\rho_h}$.
\end{proof}

Note that $C$ is a symmetric partial $H$-comodule coalgebra via $\overline{\rho_h}$ if and only if
\begin{itemize}
	\item [(i)] $\varepsilon_H(h)=1_{\Bbbk}$
	\item [(ii)] $(h \otimes 1_H)\Delta_H(h)=h \otimes h$
	\item [(iii)] $\Delta_H(h)(h\otimes 1_H)=h \otimes h.$
\end{itemize}

\begin{obs}\label{obscarac}
If in addition $h \in H$ satisfies $h = \varepsilon_t(h)$, then $C$ is a $H$-comodule coalgebra.
\end{obs}

If $\Bbbk {G}$ is the weak Hopf algebra given in Example \ref{exemplodogrupo} and $h$ is an element in $G$, then every parcial $\Bbbk {G}$-comodule coalgebra via $\rho_h$ is actually a (global) $\Bbbk {G}$-comodule coalgebra via $\rho_h$ thanks to Remark \ref{obscarac}. The following example was inspired by Example $3.2.3$ of \cite{glaubertese}, where $\Bbbk$ is seen as a partial $\Bbbk \mathcal{G}$-comodule algebra.

\begin{ex}
Consider $\Bbbk \mathcal{G}$ the groupoid algebra where the groupoid $\mathcal{G}$ is the disjoint union of the finites groups $G_1$ and $G_2$. Under these conditions, any coalgebra $C$ is a partial $\Bbbk \mathcal{G}$-comodule coalgebra via $\overline{\rho_h}$ where
	$h = \displaystyle \sum_{g \in G_1} \frac{1}{|G_1|}\delta_g.$
\end{ex}

We can also characterize the coaction of the weak Hopf algebra $\Bbbk \mathcal{G}$ on $\Bbbk$ when $\mathcal{G} $ a finite groupoid. 

\begin{ex}$\Bbbk$ is a partial $\Bbbk \mathcal{G}$-comodule coalgebra via $\overline{\rho_h}$ if and only if $h = \displaystyle \sum_{g \in N} \dfrac{1}{|N|}\delta_g$, for $N$ some group in $\mathcal{G}$.
\end{ex}

\begin{ex}
	$\Bbbk$ is a partial $(\Bbbk \mathcal{G})^*$-comodule coalgebra via $\overline{\rho_f}$ if and only if $f = \displaystyle \sum_{g \in N} p_g$, where $N$ is a group in $\mathcal{G}$.
\end{ex}

\subsection{Induced Coaction} \label{induzidacoacao}
\quad
Let $H$ be a weak Hopf algebra and $C$ a coalgebra. Suppose that $C$ is a $H$-comodule coalgebra via
\begin{eqnarray*}
	\rho: C & \longrightarrow & H\otimes C\\
	c & \longmapsto & c^{-1} \otimes c^0
\end{eqnarray*}
Our goal in this section is to construct a symmetric partial $H$-comodule coalgebra from a $H$-comodule coalgebra. For this, consider $D \subseteq C$ a subcoalgebra of $C$ such that there exists a projection $\pi:C \rightarrow C$ onto $D$, i.e.,
$\pi(\pi(c)) = \pi(c)$ for all $c \in C, $ and $Im \pi = D.$ Under these conditions, it can be obtained the following result.

\begin{prop} \label{coinduzida}
	$D$ is a symmetric partial $H$-comodule coalgebra via
	$$\begin{array}{rl}
	\overline{\rho}: D & \longrightarrow  H\otimes D\\
	d  & \longmapsto (I_H \otimes \pi)\rho(d)=d^{-1} \otimes \pi(d^0)
	\end{array}$$
	if and only if the projection $\pi$ satisfies:
	\begin{itemize}
		\item [(i)]$d^{-1} \otimes \Delta_D(\pi(d^0)) = d^{-1} \otimes (\pi \otimes \pi) (\Delta_D(d^0))$;
		\item [(ii)]$d^{-1} \otimes {\pi(d^0)}^{-1} \otimes \pi({\pi(d^0)}^{0}) = {d_1}^{-1} \varepsilon_D(\pi({d_1}^{0})) {{d_2}^{-1}}_1\otimes {{d_2}^{-1}}_2 \otimes \pi({{d_2}^{0}})\\
{\hspace{4,1 cm}}		={{d_1}^{-1}}_1{d_2}^{-1} \varepsilon_D(\pi({d_2}^{0})) \otimes {{d_1}^{-1}}_2 \otimes \pi({{d_1}^{0}})$;
		\end{itemize}
for all $d \in D$. In this case we say that $\overline{\rho}$ is an induced coaction.
\end{prop}

\begin{proof}
	Suppose that $D$ is a symmetric partial $H$-comodule coalgebra via
	$$\overline{\rho}(c)= (I_H \otimes \pi)\rho(d)=d^{-1} \otimes \pi(d^0).$$
	Therefore,
\begin{eqnarray*}
			d^{-1} \otimes \Delta_D(\pi(d^0)) &=& (I_H \otimes \Delta_D)(\overline{\rho}(d))\\
			&\stackrel{(CCP2)}{=}& (m_H \otimes I_D \otimes I_D)(I_H \otimes \tau_{H , D} \otimes I_D)(\overline{\rho} \otimes \overline{\rho}) \Delta_D(d)\\
			&\stackrel{(CC2)}{=}&d^{-1} \otimes (\pi \otimes \pi) (\Delta_D(d^0)).
		\end{eqnarray*}
Besides that,	
 \begin{eqnarray*}
			d^{-1} \otimes {\pi(d^0)}^{-1} \otimes \pi({\pi(d^0)}^{0}) &=&
			(I_H \otimes \overline{\rho})\overline{\rho}(d)\\
			&\stackrel{(CCP3)}{=}&(m_H \otimes I_H \otimes I_D)[(I_H \otimes \varepsilon_D)(\overline{\rho}(d_1)) \otimes (\Delta_H \otimes I_D)(\overline{\rho}(d_2))]\\
			&=&{d_1}^{-1} \varepsilon_D(\pi({d_1}^{0})) {{d_2}^{-1}}_1\otimes {{d_2}^{-1}}_2 \otimes \pi({{d_2}^{0}}).
		\end{eqnarray*}
		
		Analogously, using the symmetry condition,
		\begin{eqnarray*}
			d^{-1} \otimes {\pi(d^0)}^{-1} \otimes \pi({\pi(d^0)}^{0}) &=&
			(I_H \otimes \overline{\rho})\overline{\rho}(d)\\
		&=&{{d_1}^{-1}}_1{d_2}^{-1} \varepsilon_D(\pi({d_2}^{0})) \otimes {{d_1}^{-1}}_2 \otimes \pi({{d_1}^{0}}).
		\end{eqnarray*}
The converse is immediate.
\end{proof}
 
Note that the induced coaction is a $H$-comodule coalgebra if and only if $\varepsilon_t(d^{-1})\varepsilon_D(\pi(d^0)) = d^{-1}\varepsilon_D(\pi(d^0))$ for all $d \in D$.

\begin{ex}
	Consider $G_1$ and $G_2$ finite groups, $\mathcal{G}$ the groupoid generated by the disjoint union of these groups and $\Bbbk \mathcal{G}$ its groupoid algebra. Define
	\begin{eqnarray*}
		\rho: \Bbbk G_1 & \longrightarrow &  ({\Bbbk \mathcal{G}})^* \otimes \Bbbk G_1 \\
		h & \longmapsto & \displaystyle \sum_{g \in G_1}  p_{g} \otimes hg
	\end{eqnarray*}
	with $\{g\}_{g \in G_1}$ basis for the Hopf algebra $\Bbbk G_1$ and $\{p_g\}_{g \in \mathcal{G}}$ the dual basis for the weak Hopf algebra $({\Bbbk \mathcal{G}})^*$.
	Thus, $\Bbbk G_1$ is a $({\Bbbk \mathcal{G}})^*$-comodule coalgebra. Define
	\begin{eqnarray*}
		\pi: \Bbbk {G_1} &\rightarrow& \Bbbk {G_1}\\
		g &\mapsto& \left\{
		\begin{array}{rl}
			g, & \text{if $g \in D$ },\\
			0, & \text{ otherwise, }
		\end{array} \right.
	\end{eqnarray*}
	where $D = <l>_{\Bbbk}$, for some $l$ fixed in $G_1$. Then, $D$ is a symmetric partial $({\Bbbk \mathcal{G}})^*$-comodule coalgebra by Proposition \ref{coinduzida}. Moreover, the induced coaction constructed is not global. Indeed, on the one hand,
\begin{eqnarray*}
	{h}^{-1}\varepsilon_D(\pi({h}^0))&=& \sum_{g \in G_1} p_g \varepsilon_D(\pi({hg}))\\
	&\stackrel{hg=l}{=}&p_{h^{-1}l}.
\end{eqnarray*}

On the other hand,
\begin{eqnarray*}
	{\varepsilon_t}_{_{(\Bbbk \mathcal{G})^*}}({h}^{-1})\varepsilon_D(\pi({h}^0))&=& \sum_{g \in G_1} {\varepsilon_t}_{_{\Bbbk \mathcal{G}^*}}(p_g) \varepsilon_D(\pi({hg}))\\
	&\stackrel{hg=l}{=}& {\varepsilon_t}_{_{(\Bbbk \mathcal{G})^*}}(p_{h^{-1}l}).
\end{eqnarray*}

Note that $p_{h^{-1}l} \neq {\varepsilon_t}_{_{(\Bbbk \mathcal{G})^*}}(p_{h^{-1}l})$. Therefore,  $D$ is not a global $(\Bbbk \mathcal{G})^*$-comodule coalgebra.

\end{ex}

\section{Weak Smash Coproduct}
\quad \
Yu. Wang and L. Yu. Zhang, in \cite{Wang}, introduced a new structure generated from a $H$- comodule coalgebra, the weak smash coproduct. Therefore, a natural question arises: ``\textit{Under what conditions the weak smash coproduct becomes a weak Hopf algebra?}" 

This section is destined to answer this question and to make a contribution to the existing theory. The weak smash coproduct was defined being the vector space $C \times H = <c^0 \otimes h_2 \varepsilon_H(c^{-1}h_1)>_{\Bbbk}$ that has a specific structure of coalgebra. Initially we start showing that this structure of coalgebra is inherited from some properties of the vector space $C \otimes H$.

\begin{prop}
	The vector space $C \otimes H$ has a coassociative coproduct  
	$$\Delta(c \otimes h)= c_1 \otimes {c_2}^{-1}h_1 \otimes {c_2}^{0} \otimes h_2.$$
	Moreover, if $C$ is an algebra, then $C \otimes H$ is an algebra with product given by
	$$(c \otimes h)(d \otimes k)= (cd \otimes hk),$$
	and unit
	$1_{(C \otimes H)}=1_{C} \otimes 1_H.$
\end{prop}

\begin{proof}Indeed for all $c \in C$ and $h \in H$
	\begin{eqnarray*}
		(I \otimes \Delta)\Delta(c \otimes h)
		&=&c_1 \otimes {c_2}^{-1}h_1 \otimes {{c_2}^{0}}_1 \otimes {{{c_2}^{0}}_2}^{-1} h_2 \otimes  {{{c_2}^{0}}_2}^{0} \otimes h_3\\
		&\stackrel{(CC2)}{=}&c_1 \otimes {c_2}^{-1}{c_3}^{-1}h_1 \otimes {{c_2}^{0}} \otimes {{{c_3}^{0-1}}} h_2 \otimes  {{{c_3}^{00}}} \otimes h_3\\
		&\stackrel{(CC3)}{=}&c_1 \otimes {c_2}^{-1}{{c_3}^{-1}}_1h_1 \otimes {{c_2}^{0}} \otimes {{c_3}^{-1}}_2 h_2 \otimes  {{{c_3}^{0}}} \otimes h_3\\
		&=&(\Delta \otimes I)\Delta(c \otimes h).
	\end{eqnarray*}	
	
The properties of associativity and unit follow naturally.\end{proof}


\begin{prop}
Let $H$ a weak Hopf algebra and $C$ a $H$-comodule coalgebra. Then, the vector space $C \times H = \ \  <c^0 \otimes h_2 \varepsilon_H(c^{-1}h_1)>_{\Bbbk}$ is a coalgebra with counit $\varepsilon(c \times h)= \varepsilon_C(c^{0})\varepsilon_H(c^{-1}h).$ \end{prop}
\begin{proof}
	Indeed for all $c \in C$ and $h \in H$
	\begin{eqnarray*}
	\Delta(c \times h)&=& \Delta(c^{0} \otimes h_2 \varepsilon_H({c^{-1}}h_1))\\
	&=&{c^0}_1 \otimes {{c^0}_2}^{-1}{h_2}\varepsilon_H({c}^{-1} h_1) \otimes {{c^0}_2}^{0} \otimes {h_3} \\
	&\stackrel{(CC2)}{=}&{c_1}^0 \otimes {{c_2}^{0-1}}{h_2}\varepsilon_H({c_1}^{-1} {{c_2}^{-1}} h_1) \otimes {c_2}^{00} \otimes {h_3} \\
	&\stackrel{(CC3)}{=}&{c_1}^0 \otimes {{c_2}^{-1}}_2{h_2}\varepsilon_H({c_1}^{-1} {{c_2}^{-1}}_1 h_1) \otimes {c_2}^{0} \otimes {h_3} \\
	&=&{c_1}^0 \otimes {({{c_2}^{-1}}_1)}_2{h_2} \varepsilon_H({c_1}^{-1} {({{c_2}^{-1}}_1)}_1 h_1) \otimes {c_2}^{0} \otimes {h_4} \varepsilon_H({{c_2}^{-1}}_2{h_3})\\
	&\stackrel{(CC3)}{=}&{c_1}^0 \otimes {{c_2}^{-1}}_2{h_2} \varepsilon_H({c_1}^{-1} {{c_2}^{-1}}_1 h_1) \otimes {c_2}^{00} \otimes {h_4} \varepsilon_H({c_2}^{0-1}{h_3})\\
	&=& c_1 \times {c_2}^{-1}h_1 \otimes {c_2}^{0} \times h_2.
	\end{eqnarray*}
	Besided that, 
	$C \times H$ is counitary since
	\begin{eqnarray*}
		(\varepsilon \otimes I)\Delta(c \times h)&=&(\varepsilon \otimes I)(\Delta(c^{0} \otimes h_2 \varepsilon_H({c^{-1}}h_1)))\\
		&=&(\varepsilon \otimes I)({c^0}_1 \otimes {{c^0}_2}^{-1}{h_2}\varepsilon_H({c}^{-1} h_1) \otimes {{c^0}_2}^{0} \otimes {h_3})\\
		&=&\varepsilon_C({c^0}_1) \varepsilon_H({{c^0}_2}^{-1}{h_2})\varepsilon_H({c}^{-1} h_1) ({{c^0}_2}^{0} \otimes {h_3})\\
		&=&\varepsilon_C(c^{0-1}{h_2})\varepsilon_H({c}^{-1} h_1) (c^{00} \otimes {h_3})\\
		&\stackrel{(CC3)}{=}&\varepsilon_H({c^{-1}}_2{h_2})\varepsilon_H({c^{-1}}_1 h_1) (c^{0} \otimes {h_3})\\
		&=&\varepsilon_H({c^{-1}} h_1) (c^{0} \otimes {h_2})\\
		&=&c \times h.
	\end{eqnarray*}
	
	Similarly,
	\begin{eqnarray*}
		(I \otimes \varepsilon)\Delta(c \times h)
		&=&({c^0}_1 \otimes {{c^0}_2}^{-1}{h_2})\varepsilon_H({c}^{-1} h_1) \varepsilon_C({{c^0}_2}^{0}) \varepsilon_H({h_3})\\
	&=&({c^0}_1 \otimes {{c^0}_2}^{-1}{h_2})\varepsilon_H({c}^{-1} h_1) \varepsilon_C({{c^0}_2}^{0}) \\
		&\stackrel{(CC2)}{=}&({c_1}^0 \otimes {c_2}^{0-1}{h_2})\varepsilon_H({c_1}^{-1}{c_2}^{-1} h_1) \varepsilon_C({c_2}^{00}) \\
		&\stackrel{(CC3)}{=}&({c_1}^0 \otimes {{c_2}^{-1}}_2{h_2})\varepsilon_H({c_1}^{-1}{{c_2}^{-1}}_1 h_1) \varepsilon_C({c_2}^{0}) \\
		&{=}&({c_1}^0 \otimes {\varepsilon_t({c_2}^{-1}})_2{h})\varepsilon_H({c_1}^{-1}\varepsilon_t({{c_2}^{-1}})_1)  \varepsilon_C({c_2}^{0}) \\
		&\stackrel{(\ref{4.10})}{=}&({c_1}^0 \otimes {1_H}_2{h})\varepsilon_H({c_1}^{-1}{1_H}_1\varepsilon_t({c_2}^{-1}))  \varepsilon_C({c_2}^{0}) \\
		&\stackrel{(\ref{4.16})}{=}&({c_1}^0 \otimes {1_H}_2{h})\varepsilon_H({c_1}^{-1}\varepsilon_t({c_2}^{-1}){1_H}_1)\varepsilon_C({c_2}^{0}) \\
		&\stackrel{(\ref{4.7})}{=}&({c^0} \otimes {1_H}_2{h})\varepsilon_H({c}^{-1}\varepsilon_s({1_H}_1))  \\
		&\stackrel{(\ref{4.5})}{=}&({c^0} \otimes {1_H}_2{h})\varepsilon_H({c}^{-1} \varepsilon_t(\varepsilon_s({1_H}_1)))  \\
		&\stackrel{(\ref{4.34})}{=}&({c^0} \otimes {1_H}_2{h})\varepsilon_H({c}^{-1} \varepsilon_t(S_H({1_H}_1)))  \\
		&\stackrel{(\ref{4.5})}{=}&({c^0} \otimes {1_H}_2{h})\varepsilon_H({c}^{-1} S_H({1_H}_1))  \\
		&\stackrel{(\ref{4.39})}{=}&({c^0} \otimes h_2)\varepsilon_H({c}^{-1} \varepsilon_t(h_1))  \\
		&\stackrel{(\ref{4.5})}{=}&({c^0} \otimes h_2)\varepsilon_H({c}^{-1}h_1)  \\
		&=&c \times h.
	\end{eqnarray*}

	Therefore, $C \times H = <c^0 \otimes h_2 \varepsilon_H(c^{-1}h_1)>_{\Bbbk}$ is a coalgebra.
\end{proof}

\begin{defn}
Let $C$ a weak bialgebra and $H$ a weak Hopf algebra. We say that $C$ is a (left) $H$-\textit{comodule bialgebra} if there exists a linear map $\rho: C \rightarrow H \otimes C$ such that $\rho$ defines simultaneously a structure of (left) $H$-comodule coalgebra and (left) $H$-comodule algebra in $C$.
\end{defn}

From now, suppose that $C$ is a $H$-comodule bialgebra (then $C$ is a weak bialgebra) where $H$ is a commutative weak Hopf algebra. It follows from the commutativity of $H$ that $\varepsilon_t$ and $\varepsilon_s$ are multiplicatives. Moreover, S. Caenepeel and E. De Groot showed in \cite{Caenepeel} that
\begin{eqnarray}
\varepsilon_t \circ \overline{\varepsilon_t} &=& \varepsilon_t\\
\varepsilon_s \circ \overline{\varepsilon_s} &=& \varepsilon_s
\end{eqnarray}
where $\overline{\varepsilon_t}(h) = {1_H}_1 \varepsilon_H({1_H}_2 h)$ and $\overline{\varepsilon_s}(h) = {1_H}_2 \varepsilon_H(h{1_H}_1)$ for all $h \in H$. Since $H$ is commutative
\begin{eqnarray*}
	\overline{\varepsilon_t}(h) &=& {1_H}_1 \varepsilon_H({1_H}_2 h)\\
	&=& {1_H}_1 \varepsilon_H(h{1_H}_2)\\
	&=&\varepsilon_s(h).
\end{eqnarray*}

And, analogously $\overline{\varepsilon_s}(h) =\varepsilon_t(h).$ Therefore, 
\begin{eqnarray}
\varepsilon_t \circ \varepsilon_s &=& \varepsilon_t \label{epsilon1}\\
\varepsilon_s \circ \varepsilon_t &=& \varepsilon_s \label{epsilon2}.
\end{eqnarray}

These properties for $\varepsilon_t$ and $\varepsilon_s$ will be commonly used throughout this section.

\begin{lema} \label{produto}
The product in $C \otimes H$ induces a product in $C \times H$ given by
	$(c \times h)(b \times k) = (cb \times hk),$
	for all $c \in C$ and $h \in H$.
\end{lema}

\begin{proof}
	Let $c, b \in C$ and $h,k \in H$. Then,
	\begin{eqnarray*}
		(c \times h)(b \times k)&=&(c^{0} \otimes h_2 \varepsilon_H(c^{-1}h_1))(b^{0} \otimes k_2 \varepsilon_H(b^{-1}k_1))  \\
		&\stackrel{(\ref{4.5})}{=}& c^{0}b^{0} \otimes h_2 k_2 \varepsilon_H(c^{-1}\varepsilon_t(h_1))\varepsilon_H(b^{-1}\varepsilon_t(k_1))  \\
		&\stackrel{(\ref{4.39})}{=}& c^{0}b^{0} \otimes {1_H}_2 h {1_H}_{2'}k \varepsilon_H(c^{-1}S_H({1_H}_1))\varepsilon_H(b^{-1}S_H({1_H}_{1'}))  \\
		&\stackrel{(\ref{4.5})}{=}& c^{0}b^{0} \otimes {1_H}_2 h {1_H}_{2'}k \varepsilon_H(c^{-1}\varepsilon_t(S_H({1_H}_1)))\varepsilon_H(b^{-1}\varepsilon_t(S_H({1_H}_{1'})))  \\
		&\stackrel{(\ref{4.34})}{=}& c^{0}b^{0} \otimes {1_H}_2 h {1_H}_{2'}k \varepsilon_H(c^{-1}\varepsilon_t(\varepsilon_s({1_H}_1)))\varepsilon_H(b^{-1}\varepsilon_t(\varepsilon_s({1_H}_{1'})))  \\
		&\stackrel{(\ref{4.7})}{=}& c^{0}b^{0} \otimes {1_H}_2 h {1_H}_{2'}k \varepsilon_H(c^{-1}\varepsilon_t({1_H}_1))\varepsilon_H(b^{-1}\varepsilon_t({1_H}_{1'}))  \\
		&\stackrel{(\ref{4.5})}{=}& c^{0}b^{0} \otimes {1_H}_2 h {1_H}_{2'}k \varepsilon_H(c^{-1}{1_H}_1)\varepsilon_H(b^{-1}{1_H}_{1'})  \\
		&\stackrel{(\ref{4.7})}{=}&  {(cb)}^{0} \otimes {1_H}_2\varepsilon_H(\varepsilon_s({1_H}_1){(cb)}^{-1}) h k   \\
		&\stackrel{(\ref{4.5})}{=}&  {(cb)}^{0} \otimes {1_H}_2\varepsilon_H({(cb)}^{-1}\varepsilon_t(\varepsilon_s({1_H}_1))) h k   \\
		&\stackrel{(\ref{4.34})}{=}&  {(cb)}^{0} \otimes {1_H}_2\varepsilon_H({(cb)}^{-1}\varepsilon_t(S_H({1_H}_1))) h k   \\
		&\stackrel{(\ref{4.5})}{=}&  {(cb)}^{0} \otimes {1_H}_2\varepsilon_H({(cb)}^{-1}S_H({1_H}_1)) h k   \\
		&\stackrel{(\ref{4.39})}{=}&  {(cb)}^{0} \otimes {(hk)}_2 \varepsilon_H({(cb)}^{-1}\varepsilon_t({(hk)}_1))   \\
		&=& (cb \times hk).
	\end{eqnarray*}
\end{proof}

Note that the associativity of the algebras $C$ and $H$ guarantee the associativity of $C \times H$. We also have that $1_{C \times H} = 1_C \times 1_H$.

For the next result consider $C \times H$ with the product given in Lemma \ref{produto} and the coproduct given by $\Delta(c \times h) = c_1 \times {c_2}^{-1}h_1 \otimes {c_2}^{0} \times h_2$ for all $c \in C$ and $h \in H$.

\begin{lema} \label{delta}
	Under these conditions, the coproduct defined on ${C \times H}$ is multiplicative.	
\end{lema}

\begin{proof}
	
	Given $b, c \in C$ and $h, k \in H$
	\begin{eqnarray*}
		\Delta((c \times h)(b \times k))&=& \Delta (cb \times hk)\\
		&=&{(cb)}_1 \times {{(cb)}_2}^{-1}{(hk)}_1 \otimes {{(cb)}_2}^{0} \times {(hk)}_2 \\
		&=&{c}_1{b}_1 \times {({c}_2{b}_2)}^{-1}{h}_1{k}_1 \otimes {({c}_2{b}_2)}^{0} \times {h}_2{k}_2 \\
		&{=}&{c}_1{b}_1 \times {{c}_2}^{-1}{{b}_2}^{-1}{h}_1{k}_1 \otimes {{c}_2}^{0}{{b}_2}^{0} \times {h}_2{k}_2 \\
		&=&({c}_1  \times {{c}_2}^{-1}{h}_1)({b}_1 \times {{b}_2}^{-1}{k}_1) \otimes ({{c}_2}^{0} \times {h}_2)({{b}_2}^{0} \times {k}_2) \\
		&=&[({c}_1  \times {{c}_2}^{-1}{h}_1) \otimes ({{c}_2}^{0} \times {h}_2)][({b}_1 \times {{b}_2}^{-1}{k}_1) \otimes ({{b}_2}^{0} \times {k}_2)] \\
		&=&\Delta(c \times h) \Delta (b \times k).
	\end{eqnarray*}	
	
	\end{proof}

Once proved the multiplicativity of the coproduct of ${C \times H}$ and knowing that the coassociativity follows by the fact that the coproduct of the weak smash coproduct is inherited from $C \otimes H$, we are able to show the properties of weak bialgebra for $\varepsilon_{C \times H}$ as we can see in the following lemma.

\begin{lema}\label{epslon}
Let $C \times H$ be the weak smash coproduct. Then, the counity $\varepsilon_{C \times H}$ satisfies
	\begin{eqnarray*}
		\varepsilon((a \times k)(c \times h)(b \times l))&=&\varepsilon((a \times k){(c \times h)}_1) \varepsilon({(c \times h)}_2(b \times l))\\
		&=&\varepsilon((a \times k){(c \times h)}_2) \varepsilon({(c \times h)}_1(b \times l))
	\end{eqnarray*}
	for all $a,b,c \in C$ and $h, k, l \in H$. 	
\end{lema}

\begin{proof}
	Let $a,b,c \in C$ and $h, k, l \in H$,
	\begin{eqnarray*}
		\varepsilon((a \times k)(c \times h)(b \times l))&=& \varepsilon(acb \times khl)\\
		&=&\varepsilon_C({(acb)}^0)\varepsilon_H({(acb)}^{-1}khl)\\
		&=&\varepsilon_C(a^0c^0b^0)\varepsilon_H(a^{-1}c^{-1}b^{-1}khl)
	\end{eqnarray*}
	
	On the one hand,
	\begin{eqnarray*}
		&\ &\varepsilon((a \times k){(c \times h)}_1) \varepsilon({(c \times h)}_2(b \times l))\\
		&=&\varepsilon_C({(ac_1)}^{0})\varepsilon_H({(ac_1)}^{-1} k{c_2}^{-1}h_1) \varepsilon_C({({c_2}^0 b)}^{0})\varepsilon_H({({c_2}^0 b)}^{-1} h_2l)\\
		&=&\varepsilon_C(a^0{c_1}^{0})\varepsilon_H(a^{-1}{c_1}^{-1} k{c_2}^{-1}h_1) \varepsilon_C({c_2}^{00} b^{0})\varepsilon_H({c_2}^{0-1} b^{-1} h_2l)\\
		&\stackrel{(CC3)}{=}&\varepsilon_C(a^0{c_1}^{0})\varepsilon_H(a^{-1}{c_1}^{-1} k{{c_2}^{-1}}_1 h_1) \varepsilon_C({c_2}^{0} b^{0})\varepsilon_H({{c_2}^{-1}}_2 b^{-1} h_2l)\\
		&\stackrel{(CC2)}{=}&\varepsilon_C(a^0{c^0}_{1})\varepsilon_H(a^{-1}{c}^{-1}k h b^{-1}l) \varepsilon_C({c^0}_{2} b^{0})\\
		&{=}&\varepsilon_C(a^0{c^0} b^{0})\varepsilon_H(a^{-1}{c}^{-1}b^{-1}khl). \\
	\end{eqnarray*}
On the other hand,
	\begin{eqnarray*}
		&\ &\varepsilon((a \times k){(c \times h)}_2) \varepsilon({(c \times h)}_1(b \times l))\\
		&=&\varepsilon(a{c_2}^0 \times kh_2) \varepsilon({c_1} b \times {c_2}^{-1}h_1l)\\
		&=&\varepsilon_C({(a{c_2}^0)}^0)\varepsilon_H({(a{c_2}^0)}^{-1} kh_2) \varepsilon_C({({c_1} b)}^{0})\varepsilon_H({({c_1} b)}^{-1} {c_2}^{-1}h_1l)\\
		&=&\varepsilon_C(a^0{c_2}^{00})\varepsilon_H(a^{-1}{c_2}^{0-1} kh_2) \varepsilon_C({c_1}^0 b^{0})\varepsilon_H({c_1}^{-1} b^{-1} {c_2}^{-1}h_1l)\\
		&\stackrel{(CC3)}{=}&\varepsilon_C(a^0{c_2}^{0})\varepsilon_H(a^{-1}{{c_2}^{-1}}_2 kh_2) \varepsilon_C({c_1}^0 b^{0})\varepsilon_H({c_1}^{-1} b^{-1} {{c_2}^{-1}}_1 h_1l)\\
		&\stackrel{(CC2)}{=}&\varepsilon_C(a^0{c^0}_{2})\varepsilon_H(a^{-1}khl{c}^{-1} b^{-1}) \varepsilon_C({c^0}_1 b^{0})\\
		&{=}&\varepsilon_C(a^0{c^0} b^{0})\varepsilon_H(a^{-1}{c}^{-1}b^{-1}khl). 
	\end{eqnarray*}
\end{proof}

\begin{lema} \label{fraca}
Let $C \times H$ be the weak smash coproduct. Then, 
	\begin{eqnarray*}
		(1_{C \times H} \otimes \Delta(1_{C \times H}))(\Delta(1_{C \times H}) \otimes 1_{C \times H}) &=& (\Delta(1_{C \times H}) \otimes 1_{C \times H})(1_{C \times H} \otimes \Delta(1_{C \times H}))\\
		&=& (I \otimes \Delta)\Delta(1_{C \times H})\\
		&\stackrel{\ref{delta}}{=}& (\Delta \otimes I)\Delta(1_{C \times H}).
	\end{eqnarray*}
\end{lema}

\begin{proof}
	Indeed,
	\begin{eqnarray*}
		&\ &(\Delta(1_{C \times H}) \otimes 1_{C \times H})(1_{C \times H} \otimes \Delta(1_{C \times H}))\\
		&=&{{1_C}_1}^0 \otimes {{{1_C}_2}^{-1}}_2 {1_H}_2 \varepsilon_H({{1_C}_1}^{-1}{{{1_C}_2}^{-1}}_1 {1_H}_1) \otimes \\
		&\ &{{1_C}_2}^{00}{{1_C}_{1'}}^0 \otimes {1_H}_4{{{1_C}_{2'}}^{-1}}_2 \varepsilon_H({{1_C}_2}^{0-1}{{1_C}_{1'}}^{-1}{1_H}_3{{{1_C}_{2'}}^{-1}}_1) \otimes {{1_C}_{2'}}^{00}\otimes {1_H}_{6} \varepsilon_H({{1_C}_{2'}}^{0-1}{1_H}_{5})\\
		&=&{{1_C}_1}^0 \otimes {{{1_C}_2}^{-1}}_3 {1_H}_2  \varepsilon_H({{1_C}_1}^{-1}{{{1_C}_2}^{-1}}_1)\varepsilon_H({{{1_C}_2}^{-1}}_2{1_H}_1)  \otimes \\
		&\ &{{1_C}_2}^{00}{{1_C}_{1'}}^0 \otimes {1_H}_4{{{1_C}_{2'}}^{-1}}_2 \varepsilon_H({{1_C}_2}^{0-1}{{1_C}_{1'}}^{-1}{1_H}_3{{{1_C}_{2'}}^{-1}}_1) \otimes {{1_C}_{2'}}^{00}\otimes {1_H}_{6} \varepsilon_H({{1_C}_{2'}}^{0-1}{1_H}_{5})\\
		&\stackrel{(CC3)}{=}&{{1_C}_1}^0 \otimes {{{1_C}_2}^{-1}}_2 {1_H}_1 \varepsilon_H({{1_C}_1}^{-1}{{{1_C}_2}^{-1}}_1)  \otimes \\
		&\ &{{1_C}_2}^{0}{{1_C}_{1'}}^0 \otimes {1_H}_2{{{1_C}_{2'}}^{-1}}_2 \varepsilon_H({{{1_C}_2}^{-1}}_3{{1_C}_{1'}}^{-1}{{{1_C}_{2'}}^{-1}}_1)  \otimes {{1_C}_{2'}}^{0}\otimes {1_H}_{4} \varepsilon_H({{{1_C}_{2'}}^{-1}}_3{1_H}_{3})\\
		&\stackrel{(\ref{4.5})}{=}&{{1_C}_1}^0 \otimes {{{1_C}_2}^{-1}}_2 {1_H}_1 \varepsilon_H({{1_C}_1}^{-1}{{{1_C}_2}^{-1}}_1)  \otimes \\
		&\ &{{1_C}_2}^{0}{{1_C}_{1'}}^0 \otimes {1_H}_2{{{1_C}_{2'}}^{-1}}_2 \varepsilon_H({{1_C}_{1'}}^{-1}{{{1_C}_{2'}}^{-1}}_1\varepsilon_t({{{1_C}_2}^{-1}}_3))  \otimes {{1_C}_{2'}}^{0}\otimes {1_H}_{3} \\
		&\stackrel{(\ref{4.12})}{=}&{{1_C}_1}^0 \otimes {1_H}_{1'}{{{1_C}_2}^{-1}}_2 {1_H}_1 \varepsilon_H({{1_C}_1}^{-1}{{{1_C}_2}^{-1}}_1)  \otimes {{1_C}_2}^{0}{{1_C}_{1'}}^0 \otimes {1_H}_2{{{1_C}_{2'}}^{-1}}_2 \varepsilon_H({{1_C}_{1'}}^{-1}{{{1_C}_{2'}}^{-1}}_1{1_H}_{2'})  \otimes {{1_C}_{2'}}^{0}\otimes {1_H}_{3} \\
		&=&{{1_C}_{1}}^0 \otimes \varepsilon_s({1_H}_{1'}{{{1_C}_{1'}}^{-1}}_1{{{1_C}_{2'}}^{-1}}_1)\varepsilon_H({1_H}_{2'}{{{1_C}_{1'}}^{-1}}_2 {{{1_C}_{2'}}^{-1}}_2){{{1_C}_{2}}^{-1}}_2{1_H}_{1} \varepsilon_H({{1_C}_{1}}^{-1} {{{1_C}_{2}}^{-1}}_1) \otimes \\
		&\ &{{1_C}_{2}}^{0} {{1_C}_{1'}}^{0} \otimes {{{1_C}_{2'}}^{-1}}_3{1_H}_{2}  \otimes {{1_C}_{2'}}^{0} \otimes{1_H}_{3}\\
		&=&{{1_C}_{1}}^0 \otimes {{{1_C}_{2}}^{-1}}_2 \varepsilon_s({1_H}_{1'}{{{1_C}_{1'}}^{-1}}_1{{{1_C}_{2'}}^{-1}}_1){1_H}_{1} \varepsilon_H({{1_C}_{1}}^{-1} {{{1_C}_{2}}^{-1}}_1) \otimes \\
		&\ &{{1_C}_{2}}^{0} {{1_C}_{1'}}^{0} \otimes {{{1_C}_{2'}}^{-1}}_3{1_H}_{2}\varepsilon_H({{{1_C}_{1'}}^{-1}}_2 {{{1_C}_{2'}}^{-1}}_2{1_H}_{2'})  \otimes {{1_C}_{2'}}^{0} \otimes{1_H}_{3}\\
		&\stackrel{(\ref{4.7})}{=}&{{1_C}_{1}}^0 \otimes {1_H}_{1'}{{{1_C}_{2}}^{-1}}_2 \varepsilon_s({{{1_C}_{1'}}^{-1}}_1{{{1_C}_{2'}}^{-1}}_1){1_H}_{1} \varepsilon_H({{1_C}_{1}}^{-1} {{{1_C}_{2}}^{-1}}_1) \otimes \\
		&\ &{{1_C}_{2}}^{0} {{1_C}_{1'}}^{0} \otimes {{{1_C}_{2'}}^{-1}}_3{1_H}_{2}\varepsilon_H({{{1_C}_{1'}}^{-1}}_2 {{{1_C}_{2'}}^{-1}}_2{1_H}_{2'})  \otimes {{1_C}_{2'}}^{0} \otimes{1_H}_{3}\\
		&\stackrel{(\ref{4.12})}{=}&{{1_C}_{1}}^0 \otimes {{{1_C}_{2}}^{-1}}_2 \varepsilon_s({{{1_C}_{1'}}^{-1}}_1{{{1_C}_{2'}}^{-1}}_1){1_H}_{1} \varepsilon_H({{1_C}_{1}}^{-1} {{{1_C}_{2}}^{-1}}_1) \otimes \\
		&\ &{{1_C}_{2}}^{0} {{1_C}_{1'}}^{0} \otimes {{{1_C}_{2'}}^{-1}}_3{1_H}_{2}\varepsilon_H({{{1_C}_{1'}}^{-1}}_2 {{{1_C}_{2'}}^{-1}}_2\varepsilon_t({{{1_C}_{2}}^{-1}}_3))  \otimes {{1_C}_{2'}}^{0} \otimes{1_H}_{3}\\
		&\stackrel{(\ref{4.5})}{=}&{{1_C}_{1}}^0 \otimes {{{1_C}_{2}}^{-1}}_2 \varepsilon_s({{{1_C}_{1'}}^{-1}}_1{{{1_C}_{2'}}^{-1}}_1){1_H}_{1} \varepsilon_H({{1_C}_{1}}^{-1} {{{1_C}_{2}}^{-1}}_1) \otimes \\
		&\ &{{1_C}_{2}}^{0} {{1_C}_{1'}}^{0} \otimes {{{1_C}_{2'}}^{-1}}_3{1_H}_{2}\varepsilon_H({{{1_C}_{1'}}^{-1}}_2 {{{1_C}_{2'}}^{-1}}_2{{{1_C}_{2}}^{-1}}_3)  \otimes {{1_C}_{2'}}^{0} \otimes{1_H}_{3}\\
		&\stackrel{(CC3)}{=}&{{1_C}_{1}}^0 \otimes {{{1_C}_{2}}^{-1}}_2 \varepsilon_s({{{1_C}_{1'}}^{-1}}{{{1_C}_{2'}}^{-1}}){1_H}_{1} \varepsilon_H({{1_C}_{1}}^{-1} {{{1_C}_{2}}^{-1}}_1) \otimes \\
		&\ &{{1_C}_{2}}^{0} {{1_C}_{1'}}^{00} \otimes {{{1_C}_{2'}}^{0-1}}_2{1_H}_{2}\varepsilon_H({{{1_C}_{2}}^{-1}}_3 {{1_C}_{1'}}^{0-1} {{{1_C}_{2'}}^{0-1}}_1)  \otimes {{1_C}_{2'}}^{00} \otimes{1_H}_{3}\\
		&\stackrel{(CC2)}{=}&{{1_C}_{1}}^0 \otimes {{{1_C}_{2}}^{-1}}_2 \varepsilon_s({{1_C}^{-1}}){1_H}_{1} \varepsilon_H({{1_C}_{1}}^{-1} {{{1_C}_{2}}^{-1}}_1) \otimes \\
		&\ &{{1_C}_{2}}^{0} {{{{1_C}^0}}_1}^0 \otimes {{{{{1_C}^0}}_2}^{-1}}_2{1_H}_{2}\varepsilon_H({{{1_C}_{2}}^{-1}}_3 {{{{1_C}^0}}_1}^{-1} {{{{{1_C}^0}}_2}^{-1}}_1)  \otimes {{{{1_C}^0}}_2}^{0} \otimes{1_H}_{3}\\
		&\stackrel{(*)}{=}&{{1_C}_{1}}^0 \otimes {{{1_C}_{2}}^{-1}}_2{{1_C}^{-1}}{1_H}_{1} \varepsilon_H({{1_C}_{1}}^{-1} {{{1_C}_{2}}^{-1}}_1) \otimes \\
		&\ &{{1_C}_{2}}^{0} {{{{1_C}^0}}_1}^0 \otimes {{{{{1_C}^0}}_2}^{-1}}_2{1_H}_{2}\varepsilon_H({{{1_C}_{2}}^{-1}}_3 {{{{1_C}^0}}_1}^{-1} {{{{{1_C}^0}}_2}^{-1}}_1)  \otimes {{{{1_C}^0}}_2}^{0} \otimes{1_H}_{3} \\
		&=&{{1_C}_{1}}^0 \otimes {{{1_C}_{2}}^{-1}}_2{{1_C}^{-1}}{1_H}_{1} \varepsilon_H({{1_C}_{1}}^{-1} {{{1_C}_{2}}^{-1}}_1) \otimes \\
		&\ &{{1_C}_{2}}^{0} {{{{1_C}^0}}_1}^0 \otimes {{{{{1_C}^0}}_2}^{-1}}_2{1_H}_{2}\varepsilon_H({{{{{1_C}^0}}_2}^{-1}}_3
		{1_H}_{3})\varepsilon_H({{{1_C}_{2}}^{-1}}_3 {{{{1_C}^0}}_1}^{-1} {{{{{1_C}^0}}_2}^{-1}}_1)  \otimes {{{{1_C}^0}}_2}^{0} \otimes{1_H}_{4} \\
		&\stackrel{(CC3)}{=}&{{1_C}_{1}}^0 \otimes {{{1_C}_{2}}^{-1}}_2{{1_C}^{-1}}{1_H}_{1} \varepsilon_H({{1_C}_{1}}^{-1} {{{1_C}_{2}}^{-1}}_1) \otimes \\
		&\ &{{1_C}_{2}}^{00} {{{{1_C}^0}}_1}^0 \otimes {{{{{1_C}^0}}_2}^{-1}}_2{1_H}_{2}\varepsilon_H({{1_C}_{2}}^{0-1} {{{{1_C}^0}}_1}^{-1} {{{{{1_C}^0}}_2}^{-1}}_1)  \otimes {{{{1_C}^0}}_2}^{00} \otimes{1_H}_{4} \varepsilon_H({{{{1_C}^0}}_2}^{0-1}
		{1_H}_{3})\\
		&\stackrel{(\ref{4.1})}{=}&{{1_C}_{1}}^0 \otimes {{{1_C}_{2}}^{-1}}_2{1_H}_{2'}{{1_C}^{-1}}{1_H}_{1} \varepsilon_H({{1_C}_{1}}^{-1} {{{1_C}_{2}}^{-1}}_1{1_H}_{1'}) \otimes \\
		&\ &{{1_C}_{2}}^{00} {{{{1_C}^0}}_1}^0 \otimes {{{{{1_C}^0}}_2}^{-1}}_2{1_H}_{2}\varepsilon_H({{1_C}_{2}}^{0-1} {{{{1_C}^0}}_1}^{-1} {{{{{1_C}^0}}_2}^{-1}}_1)  \otimes {{{{1_C}^0}}_2}^{00} \otimes{1_H}_{4} \varepsilon_H({{{{1_C}^0}}_2}^{0-1}
		{1_H}_{3})\\
		&\stackrel{(\ref{4.13})}{=}&{{1_C}_{1}}^0 \otimes {{{1_C}_{2}}^{-1}}_2{{1_C}^{-1}}_2{1_H}_{1} \varepsilon_H(\varepsilon_s({{1_C}^{-1}}_1){{1_C}_{1}}^{-1} {{{1_C}_{2}}^{-1}}_1) \otimes \\
		&\ &{{1_C}_{2}}^{00} {{{{1_C}^0}}_1}^0 \otimes {{{{{1_C}^0}}_2}^{-1}}_2{1_H}_{2}\varepsilon_H({{1_C}_{2}}^{0-1} {{{{1_C}^0}}_1}^{-1} {{{{{1_C}^0}}_2}^{-1}}_1)  \otimes {{{{1_C}^0}}_2}^{00} \otimes{1_H}_{4} \varepsilon_H({{{{1_C}^0}}_2}^{0-1}
		{1_H}_{3})\\
		&\stackrel{(\ref{4.6})}{=}&{{1_C}_{1}}^0 \otimes {{{1_C}_{2}}^{-1}}_2{{1_C}^{-1}}_2{1_H}_{1} \varepsilon_H({{1_C}^{-1}}_1{{1_C}_{1}}^{-1} {{{1_C}_{2}}^{-1}}_1) \otimes \\
		&\ &{{1_C}_{2}}^{00} {{{{1_C}^0}}_1}^0 \otimes {{{{{1_C}^0}}_2}^{-1}}_2{1_H}_{2}\varepsilon_H({{1_C}_{2}}^{0-1} {{{{1_C}^0}}_1}^{-1} {{{{{1_C}^0}}_2}^{-1}}_1)  \otimes {{{{1_C}^0}}_2}^{00} \otimes{1_H}_{4} \varepsilon_H({{{{1_C}^0}}_2}^{0-1}
		{1_H}_{3})\\
		&=&{{1_C}_{1}}^0 \otimes {{{1_C}_{2}}^{-1}}_2{{1_C}^{-1}}_3{1_H}_{2}\varepsilon_H({{1_C}^{-1}}_2{1_H}_{1}) \varepsilon_H({{1_C}_{1}}^{-1} {{{1_C}_{2}}^{-1}}_1{{1_C}^{-1}}_1) \otimes \\
		&\ &{{1_C}_{2}}^{00} {{{{1_C}^0}}_1}^0 \otimes {{{{{1_C}^0}}_2}^{-1}}_3{1_H}_{4} \varepsilon_H({{{{{1_C}^0}}_2}^{-1}}_2{1_H}_{3})\varepsilon_H({{1_C}_{2}}^{0-1} {{{{1_C}^0}}_1}^{-1} {{{{{1_C}^0}}_2}^{-1}}_1)  \otimes {{{{1_C}^0}}_2}^{00} \otimes{1_H}_{6} \varepsilon_H({{{{1_C}^0}}_2}^{0-1}
		{1_H}_{5})\\
		&=&{1_C}_{1} \times {{1_C}_{2}}^{-1}{1_C}^{-1}{1_H}_{1} \otimes {{1_C}_{2}}^{0} {{{1_C}^0}}_1\times {{{{1_C}^0}}_2}^{-1}{1_H}_{2} \otimes {{{{1_C}^0}}_2}^{0} \times{1_H}_{3}\\
		&\stackrel{(CC2)}{=}&{1_C}_{1} \times {{1_C}_{2}}^{-1}{{1_C}_{1'}}^{-1}{{1_C}_{2'}}^{-1} {1_H}_{1} \otimes {{1_C}_{2}}^{0} {{1_C}_{1'}}^{0}\times {{1_C}_{2'}}^{0-1}{1_H}_{2} \otimes {{1_C}_{2'}}^{00} \times{1_H}_{3}\\
		&{=}&{1_C}_{1} \times {{1_C}_{2}}^{-1}{{1_C}_{3}}^{-1} {1_H}_{1} \otimes {{1_C}_{2}}^{0} \times {{1_C}_{3}}^{0-1}{1_H}_{2} \otimes {{1_C}_{3}}^{00} \times{1_H}_{3}\\
		&\stackrel{(CC2)}{=}&{1_C}_{1} \times {{1_C}_{2}}^{-1} {1_H}_{1} \otimes {{{1_C}_{2}}^{0}}_1 \times {{{{1_C}_{2}}^{0}}_2}^{-1}{1_H}_{2} \otimes {{{{1_C}_{2}}^{0}}_2}^{0} \times{1_H}_{3}\\
		&=&(I \otimes \Delta)\Delta(1_{C \times H}).
	\end{eqnarray*}
In $(*)$ it was used the property $\rho(1) \in H_s \otimes C$ of Proposition 4.11 of \cite{Caenepeel}. Similarly it is possible to show
	\begin{eqnarray*}
(1_{C \times H} \otimes \Delta(1_{C \times H}))(\Delta(1_{C \times H}) \otimes 1_{C \times H})
		&=&(I \otimes \Delta)\Delta(1_{C \times H}).
	\end{eqnarray*}
\end{proof}

Therefore, we obtain the following result.

\begin{prop}\label{bialgebra fraca}
Let $C$ be a weak bialgebra and $H$ a commutative weak Hopf algebra such that $C$ is a  $H$-comodule bialgebra. Then, $C \times H$ is a weak  bialgebra.
\end{prop}

The next step is to know if we can give a weak Hopf algebra structure to $C \times H$. However, it is necessary to impose a natural condition on $C$, as it can be seen in Theorem \ref{Hopf fraca}.

\begin{teo}\label{Hopf fraca}
Let $C$ and $H$ be two weak Hopf algebras such that $C$ is a  $H$-comodule bialgebra and $H$ is commutative. Then, $C \times H$ is a weak Hopf algebra.
\end{teo}

\begin{proof}
By Proposition \ref{bialgebra fraca} we know that $C \times H$ is a weak bialgebra, so, it is enough to define a map $S_{C \times H}$ in $C \times H$ and show that $S_{C \times H}$ satisfies the properties of antipode of a weak Hopf algebra. Define $S_{C \times H}$ by
	$$S_{C \times H}(c \times h) = S_C(c^{0}) \times S_H(c^{-1}h),$$
	for all $c \in C$ and $h \in H$.

(I) ${(c \times h)}_1 S({(c \times h)}_2) = \varepsilon_t(c \times h)$, indeed for all $c \in C$ and $h \in H$
		\begin{eqnarray*}
			{(c \times h)}_1 S({(c \times h)}_2)&\stackrel{(CC3)}{=}& c_1S_C({c_2}^{0}) \times {{c_2}^{-1}}_1h_1 S_H({{c_2}^{-1}}_2h_2)\\
			&\stackrel{(\ref{4.10})}{=}& {c_1}^0 {S_C({c_2}^{0})}^0 \otimes {1_H}_2 \varepsilon_H({c_1}^{-1} {S_H({c_2}^{0})}^{-1} {1_H}_1{\varepsilon_t({{c_2}^{-1}}h)})\\
			&\stackrel{(\ref{4.5})}{=}& {c_1}^0 {S_C({c_2}^{0})}^0 \otimes \varepsilon_t({c_1}^{-1} {S_H({c_2}^{0})}^{-1} {c_2}^{-1}h)\\
			&\stackrel{(\ref{propt})}{=}& {c_1}^0 {S_C({c_2}^{0})}^0 \otimes \varepsilon_t(\varepsilon_t({{c_1}^{-1}}_1){{c_1}^{-1}}_2)  \varepsilon_t({c_2}^{-1} {S_H({c_2}^{0})}^{-1} h)\\
			&\stackrel{(\ref{4.8})}{=}& {c_1}^0 {S_C({c_2}^{0})}^0 \otimes \varepsilon_t({{c_1}^{-1}}_1)\varepsilon_t({{c_1}^{-1}}_2)  \varepsilon_t({c_2}^{-1} {S_H({c_2}^{0})}^{-1} h)\\
			&\stackrel{(CC3)}{=}&{c_1}^{00}{S_C({c_2}^0)}^0 \otimes \varepsilon_t({c_1}^{0-1}{S_H({c_2}^0)}^{-1} {c_1}^{-1}{c_2}^{-1}h)\\
			&\stackrel{(CC2)}{=}&{{c^0}_1}^0{S_C({c^0}_2)}^0 \otimes \varepsilon_t({{c^0}_1}^{-1}{S_H({c^0}_2)}^{-1} {c}^{-1}h)\\
			&\stackrel{(\ref{4.5})}{=}&{{c^0}_1}^0{S_C({c^0}_2)}^0 \otimes {1_H}_2 \varepsilon_H({1_H}_1{{c^0}_1}^{-1}{S_H({c^0}_2)}^{-1} {\varepsilon_t({c}^{-1}h)})\\
			&\stackrel{(\ref{4.10})}{=}&{{c^0}_1}^0{S_C({c^0}_2)}^0 \otimes {\varepsilon_t({c}^{-1}h)}_2 \varepsilon_H({{c^0}_1}^{-1}{S_H({c^0}_2)}^{-1} {\varepsilon_t({c}^{-1}h)}_1)\\
			&\stackrel{(CC2)}{=}&{c_1}^0S_C({c_2}^0) \times \varepsilon_t({c_1}^{-1}{c_2}^{-1}h)\\
			&\stackrel{(CC2)}{=}&{{1_C}^0} {c^0}_1S_C({c^0}_2) \times \varepsilon_t({1_C}^{-1} c^{-1}h)\\
			&\stackrel{(\ref{4.13})}{=}&\varepsilon_C(\varepsilon_s({{1_C}^0}_1) c^0) ({{1_C}^0}_2 \times \varepsilon_t({1_C}^{-1} c^{-1}h))\\
			&\stackrel{(\ref{4.6})}{=}&\varepsilon_C({{1_C}^0}_1 c^0) ({{1_C}^0}_2 \times \varepsilon_t({1_C}^{-1} c^{-1}h))\\
			&\stackrel{(CC2)}{=}&\varepsilon_C({{1_C}_1}^0 c^0) ({{1_C}_2}^0 \times \varepsilon_t({{1_C}_1}^{-1}{{1_C}_2}^{-1} c^{-1}h))\\
			&=&\varepsilon_t(c \times h).
		\end{eqnarray*}

(II) $S({(c \times h)}_1){(c \times h)}_2 = \varepsilon_s(c \times h)$, indeed for all $c \in C$ and $h \in H$
		\begin{eqnarray*}	
	S({(c \times h)}_1){(c \times h)}_2 &\stackrel{(CC2)}{=}& S_C({c^0}_1){c^0}_{2} \times S_H({c}^{-1}h_1) h_2\\
			&=& {{1_C}_1}^0\varepsilon_C({c^0}{1_C}_2) \otimes {S_H({c}^{-1})}_2 {\varepsilon_s(h)}_2 \varepsilon_H({{1_C}_1}^{-1} {S_H({c}^{-1})}_1 {\varepsilon_s(h)}_1)\\
			&\stackrel{(\ref{4.11})}{=}& {{1_C}_1}^0\varepsilon_C({c^0}{1_C}_2) \otimes {S_H({c}^{-1})}_2 {\varepsilon_s(h)}{1_H}_2 \varepsilon_H({{1_C}_1}^{-1} {S_H({c}^{-1})}_1 {1_H}_1)\\
			&=& {{1_C}_1}^0\varepsilon_C({c^0}{1_C}_2) \otimes \varepsilon_t({S_H({c}^{-1})}_1){S_H({c}^{-1})}_2 {\varepsilon_s(h)}\varepsilon_t({{1_C}_1}^{-1}) \\
			&\stackrel{(\ref{propt})}{=}& {{1_C}_1}^0\varepsilon_C({c^0}{1_C}_2) \otimes {S_H({c}^{-1})} {\varepsilon_s(h)}\varepsilon_t({{1_C}_1}^{-1}) \\
			&\stackrel{(*)}{=}& {{1_C}_1}^0\varepsilon_C({c}^0{{1_C}_2}^0) \otimes {S_H({c}^{-1}\overline{\varepsilon_t}({{1_C}_2}^{-1}))} \varepsilon_s(h)\varepsilon_t({{1_C}_1}^{-1}) \\
			&\stackrel{\overline{\varepsilon_t} = \varepsilon_s}{=}& {{1_C}_1}^0\varepsilon_C({c}^0{{1_C}_2}^0) \otimes {S_H({c}^{-1}\varepsilon_s({{1_C}_2}^{-1}))} {\varepsilon_s(h)}\varepsilon_t({{1_C}_1}^{-1}) \\
			&\stackrel{(\ref{4.34})}{=}& {{1_C}_1}^0\varepsilon_C({c}^0{{1_C}_2}^0) \otimes \varepsilon_t(\varepsilon_s({{1_C}_2}^{-1}))S_H({c}^{-1}) {\varepsilon_s(h)}\varepsilon_t({{1_C}_1}^{-1}) \\
			&\stackrel{(\ref{epsilon1})}{=}& {{1_C}_1}^0\varepsilon_C({c}^0{{1_C}_2}^0) \otimes \varepsilon_t({{1_C}_2}^{-1})S_H({c}^{-1}) {\varepsilon_s(h)}\varepsilon_t({{1_C}_1}^{-1}) \\
			&\stackrel{(CC2)}{=}& {{1_C}^0}_1\varepsilon_C({c}^0{{1_C}^0}_2) \otimes S_H({c}^{-1}){\varepsilon_s(h)}\varepsilon_t({1_C}^{-1}) \\
			&\stackrel{(\ref{epsilon1})}{=}& {{1_C}^0}_1\varepsilon_C({c}^0{{1_C}^0}_2) \otimes S_H({c}^{-1}){\varepsilon_s(h)}\varepsilon_t (\varepsilon_s({1_C}^{-1})) \\
			&\stackrel{(\ref{4.34})}{=}& {{1_C}^0}_1\varepsilon_C({c}^0{{1_C}^0}_2) \otimes S_H({c}^{-1}){\varepsilon_s(h)} S_H(\varepsilon_s({1_C}^{-1})) \\
			&\stackrel{(**)}{=}& {{1_C}^0}_1\varepsilon_C({c}^0{{1_C}^0}_2) \otimes S_H({c}^{-1}{1_C}^{-1}){\varepsilon_s(h)}  \\
			&\stackrel{(CC2)}{=}& {{1_C}_1}^0\varepsilon_C({c}^0{{1_C}_2}^0) \otimes S_H({c}^{-1}{{1_C}_1}^{-1}{{1_C}_2}^{-1}){\varepsilon_s(h)}  \\
			&\stackrel{(\ref{4.35})}{=}& {{1_C}_1}^0\varepsilon_C({c}^0{{1_C}_2}^0) \otimes S_H({{1_C}_1}^{-1})\varepsilon_s(\varepsilon_t({{1_C}_2}^{-1}))\varepsilon_s(\varepsilon_t({c}^{-1})){\varepsilon_s(h)}  \\
			&\stackrel{(\ref{epsilon2})}{=}& {{1_C}_1}^0\varepsilon_C({c}^0{{1_C}_2}^0) \otimes S_H({{1_C}_1}^{-1})\varepsilon_s({{1_C}_2}^{-1})\varepsilon_s({c}^{-1}){\varepsilon_s(h)}  \\
			&\stackrel{(CC3)}{=}& {{1_C}_1}^0\varepsilon_C({c}^0{{1_C}_2}^{00}) \otimes S_H({{1_C}_1}^{-1})S_H({{{1_C}_2}^{-1}}){{{1_C}_2}^{0-1}}\varepsilon_s({c}^{-1}h)\\
			&\stackrel{(CC2)}{=}& {{1_C}^0}_1\varepsilon_C({c}^0{{{1_C}^0}_2}^{0}) \otimes S_H({1_C}^{-1}){{{1_C}^0}_2}^{-1}\varepsilon_s({c}^{-1}h)\\
			&\stackrel{(**)}{=}& {{1_C}^0}_1\varepsilon_C({c}^0{{{1_C}^0}_2}^{0}) \otimes S_H(\varepsilon_s({1_C}^{-1})){{{1_C}^0}_2}^{-1}\varepsilon_s({c}^{-1}h)\\
			&\stackrel{(\ref{4.34})}{=}& {{1_C}^0}_1\varepsilon_C({c}^0{{{1_C}^0}_2}^{0}) \otimes \varepsilon_t(\varepsilon_s({1_C}^{-1})){{{1_C}^0}_2}^{-1}\varepsilon_s({c}^{-1}h)\\
			&\stackrel{(\ref{epsilon1})}{=}& {{1_C}^0}_1\varepsilon_C({c}^0{{{1_C}^0}_2}^{0}) \otimes \varepsilon_t({1_C}^{-1}){{{1_C}^0}_2}^{-1}\varepsilon_s({c}^{-1}h)\\
			&\stackrel{(CC2)}{=}& {{1_C}_1}^0\varepsilon_C({c}^0{{1_C}_2}^{00}) \otimes \varepsilon_t({{1_C}_1}^{-1}{{1_C}_2}^{-1}){{1_C}_2}^{0-1}\varepsilon_s({c}^{-1}h)\\
			&\stackrel{(CC3)}{=}& {{1_C}_1}^0\varepsilon_C({c}^0{{1_C}_2}^{0}) \otimes \varepsilon_t({{1_C}_1}^{-1}{{{1_C}_2}^{-1}}_1){{{1_C}_2}^{-1}}_2\varepsilon_s({c}^{-1}h)\\
			&\stackrel{(\ref{4.3})}{=}&{{1_C}_1}^0 \otimes {{{1_C}_2}^{-1}}_2 \varepsilon_s (c^{-1}h) \varepsilon_t({{{1_C}_2}^{-1}}_1) \varepsilon_t(\varepsilon_t({{1_C}_1}^{-1})) \varepsilon_C(c^0{{1_C}_2}^{0}) \\
			&\stackrel{(\ref{4.8})}{=}&{{1_C}_1}^0 \otimes {{{1_C}_2}^{-1}}_2 \varepsilon_s (c^{-1}h) \varepsilon_t({{{1_C}_2}^{-1}}_1 {{1_C}_1}^{-1}) \varepsilon_C(c^0{{1_C}_2}^{0}) \\
			&=&{{1_C}_1}^0 \otimes {{{1_C}_2}^{-1}}_2 \varepsilon_s (c^{-1}h) {1_H}_2 \varepsilon_H({{1_C}_1}^{-1} {{{1_C}_2}^{-1}}_1 {1_H}_1 ) \varepsilon_C(c^0{{1_C}_2}^{0}) \\
			&\stackrel{(\ref{4.11})}{=}&{{1_C}_1}^0 \otimes {{{1_C}_2}^{-1}}_2 {\varepsilon_s (c^{-1}h)}_2 \varepsilon_H({{1_C}_1}^{-1} {{{1_C}_2}^{-1}}_1 {\varepsilon_s (c^{-1}h)}_1 ) \varepsilon_C(c^0{{1_C}_2}^{0}) \\
			&=&{1_C}_1 \times {{{1_C}_2}^{-1}}\varepsilon_s (c^{-1}h) \varepsilon_C(c^0{{1_C}_2}^{0}) \\
			&\stackrel{(\ref{props})}{=}&{1_C}_1 \times {{{1_C}_2}^{-1}}_1\varepsilon_s({{{1_C}_2}^{-1}}_2)\varepsilon_s (c^{-1}h) \varepsilon_C(c^0{{1_C}_2}^{0}) \\
			&\stackrel{(CC3)}{=}&({1_C}_1 \times {{1_C}_2}^{-1}{1_H}_1) \varepsilon_C(c^0{{1_C}_2}^{00}) \varepsilon_H (c^{-1}{{1_C}_2}^{0-1}h{1_H}_2)\\
			&=& \varepsilon_s(c \times h).
		\end{eqnarray*}

	In $(*)$ it was used Proposition 4.27 of \cite{Caenepeel}. In $(**)$ it was used the property $\rho(1) \in H_s \otimes C$ of Proposition 4.11 of \cite{Caenepeel}.

(III) $S({(c \times h)}_1){(c \times h)}_2 S({(c \times h)}_3)= S(c \times h)$, indeed for all $c \in C$ and $h \in H$
		\begin{eqnarray*}
	S({(c \times h)}_1){(c \times h)}_2 S({(c \times h)}_3) &\stackrel{(\ref{epsilon1})}{=}& S_C({c_1}^0)\varepsilon_t({c_2}^{00}) \times S_H({c_1}^{-1}{{c_2}^{-1}} h_1)\varepsilon_t(\varepsilon_s({{c_2}^{0-1}} h_2))\\
			&\stackrel{(\ref{4.34})}{=}& S_C({c_1}^0)\varepsilon_t({c_2}^{00}) \times S_H({c_1}^{-1}{{c_2}^{-1}} h_1)S_H(\varepsilon_s({{c_2}^{0-1}} h_2))\\
			&\stackrel{(CC3)}{=}& S_C({c_1}^0)\varepsilon_t({c_2}^0) \times S_H({c_1}^{-1}{{c_2}^{-1}}_1 h_1\varepsilon_s({{c_2}^{-1}}_2)  \varepsilon_s(h_2))\\
			&\stackrel{(\ref{props})}{=}& S_C({c_1}^0)\varepsilon_t({c_2}^0) \times S_H({c_1}^{-1}{c_2}^{-1}h)\\
			&=& S(c \times h).
		\end{eqnarray*}
Therefore $C \times H$ is a weak Hopf algebra.		
	\end{proof}

\section{Dualization}

In order to show that a partial coaction on a coalgebra can generate a partial action on a coalgebra, and vice versa, we need to assume the additional hypothesis that the weak Hopf algebra $H$ is finite dimensional. Thus, we know that the dual $H^*$ of $ H $ is also a weak Hopf algebra.

\begin{teo}\label{teodualizacao}
	Let $C$ be a coalgebra and $H$ be a weak Hopf algebra finite dimensional. Then, the following affirmations are equivalent:
	\begin{itemize}
		\item [(i)] $C$ is a left partial $H$-comodule coalgebra;
		\item[(ii)] $C$ is a right partial $H^*$-module coalgebra.
	\end{itemize}
	Moreover, to say that $C$ is a left symmetric partial $H$-comodule coalgebra is equivalent to say that $C$ is a right symmetric partial $H^*$-module coalgebra.
\end{teo}

\begin{proof}
	Suppose that $C$ is a left symmetric partial $H$-comodule coalgebra via
	\begin{eqnarray*}
		\rho: C &\rightarrow& H \otimes C\\
		c &\mapsto& c^{-1} \otimes c^0.
	\end{eqnarray*}
	Then, $C$ is a right symmetric partial $H^*$-module coalgebra via
	\begin{eqnarray*}
		\leftharpoonup: C \otimes H^*&\rightarrow& C\\
		c \otimes f &\mapsto&  c \leftharpoonup f = f(c^{-1})c^0.
	\end{eqnarray*}
Conversely, suppose that $C$ is a right symmetric partial $H^*$-module coalgebra via
	\begin{eqnarray*}
		\leftharpoonup: C \otimes H^* &\rightarrow& C\\
		c \otimes f &\mapsto& c \leftharpoonup f.
	\end{eqnarray*}
	Then, $C$ is a left symmetric partial $H$-comodule coalgebra via	
	\begin{eqnarray*}
		\rho: C &\rightarrow& H \otimes C\\
		c &\mapsto& \sum_{i=1}^{n} (h_i  \otimes c \leftharpoonup {h_i}^*)
	\end{eqnarray*}
	where, $\{h_i\}_{i=1}^{n}$ is a basis of $H$ and $\{{h_i}^*\}_{i=1}^{n}$ is the dual basis of $H^*$.

\end{proof}

Let $\lambda$ be an element in $H^*$. We say that $C$ is a (right) \textit{partial $H$-module coalgebra via $\lambda$} if $c \leftharpoonup h = c \lambda(h)$ defines a partial action of $ H $ on the coalgebra $ C $ for all $h\in H$ and $c\in C$. Moreover, in \cite{EGG} it was proved that 	$C$ is a partial $H$-module coalgebra via $\lambda$ if and only if for all $h,k \in H$
\begin{enumerate}
		\item [(i)] $\lambda(1_H) = 1_\Bbbk$
		\item [(ii)] $\lambda(h)\lambda(k) =\lambda(h_1)\lambda(h_2k).$
	\end{enumerate}
Another result of dualization obtained is the one that says that the element $\lambda$ defined from the partial action via $\lambda $ is equal to the element defined from the partial coaction $\overline{\rho_\lambda}$ presented in Proposition \ref{rho_hparcial} as can be seen as follows.

\begin{cor}
	$C$ is a right partial $H$-module coalgebra via $c \leftharpoonup h = c \lambda(h)$ if and only if $C$ is a left partial $H^*$-comodule coalgebra via $\overline{\rho_{\lambda}}(c)= \lambda \otimes c$.
\end{cor}
\begin{proof}
It is enough to show that
	
	\begin{itemize}
		\item[(i)] $\varepsilon_{H^*}(\lambda) = 1_{\Bbbk}$, since $\varepsilon_{H^*}(\lambda) =\lambda(1_H).
			=1_{\Bbbk}$

		\item[(ii)] $(\lambda \otimes 1_{H^*})\Delta(\lambda)= \lambda \otimes \lambda$, since for all $h, k \in H$
				\begin{eqnarray*}
			(\lambda \otimes 1_{H^*})\Delta(\lambda)(h \otimes k)&=& (\lambda \lambda_1)(h) (\lambda_2)(k)\\
			&{=}& \lambda(h_1) \lambda_1(h_2) \lambda_2(k)\\
			&{=}& \lambda(h_1) \lambda(h_2k)\\
			&{=}& \lambda(h) \lambda(k)\\
			&=&(\lambda \otimes \lambda)(h \otimes k).
		\end{eqnarray*}
	\end{itemize}
The converse is immediate.	
\end{proof}

\end{document}